\documentclass[12pt,oneside, reqno]{amsart}
\usepackage{txfonts, bbm}
\usepackage{amsfonts}
\usepackage{mathrsfs}
\usepackage{amsmath}
\usepackage{amssymb, amsmath}
\usepackage{bbm, color, soul}
\usepackage{stmaryrd}
 \usepackage{enumerate}
\newcommand{\be}{\begin{eqnarray}}
\newcommand{\ee}{\end{eqnarray}}
\newcommand{\ce}{\begin{eqnarray*}}
\newcommand{\de}{\end{eqnarray*}}
\newtheorem{theorem}{Theorem}[section]
\newtheorem{lemma}[theorem]{Lemma}
\newtheorem{remark}[theorem]{Remark}
\newtheorem{definition}[theorem]{Definition}
\newtheorem{proposition}[theorem]{Proposition}

\newtheorem{corollary}[theorem]{Corollary}

\def\e{\varepsilon}

\def\p{\partial}

\def\[{{\Big[}}
\def\]{{\Big]}}
\def\<{{\langle}}
\def\>{{\rangle}}
\def\({{\Big(}}
\def\){{\Big)}}

\def\bx{{\mathbf{x}}}

\def\bt{\begin{theorem}}
\def\et{\end{theorem}}
\def\bl{\begin{lemma}}
\def\el{\end{lemma}}
\def\br{\begin{remark}}
\def\er{\end{remark}}
\def\bx{\begin{Example}}
\def\ex{\end{Example}}
\def\bd{\begin{definition}}
\def\ed{\end{definition}}
\def\bp{\begin{proposition}}
\def\ep{\end{proposition}}
\def\bc{\begin{corollary}}
\def\ec{\end{corollary}}

\def\cC{{\mathcal C}}

\def\cF{{\mathcal F}}

\def\cL{{\mathcal L}}

\def\mE{{\mathbbm E}}

\def\mP{{\mathbbm P}}

\def\mR{{\mathbb R}}

\def\mT{{\mathbb T}}

\def\mW{{\mathbb W}}

\def\sF{{\mathscr F}}

\def\eps{\varepsilon}

\def\geq{\geqslant}
\def\leq{\leqslant}

\def\bx{{\bf x}}
\def\R{\mathbb{R}}
\def\D{\mathbb{D}}

\def\1{{\mathbbm 1}}

\def\nin{\noindent}

\allowdisplaybreaks

\makeatletter
\@addtoreset{equation}{section}

\makeatother

\allowdisplaybreaks

\begin{document}

\title{Quantitative Homogenization of PDEs with Neumann boundary conditions: a probabilistic approach}
\author[author1]{Zhen-Qing Chen}
  \author[author2]{Jing Wu}
\dedicatory{ University of Washington and Sun Yat-sen University}
\thanks{E-mails: Z.-Q. Chen: zqchen@uw.edu; ~J. Wu: wujing38@mail.sysu.edu.cn}

\begin{abstract}
In this paper, we study quantitative homogenization   for  viscosity solutions of multi-scale  semilinear second order partial differential equations (PDEs) on convex domains with Neumann boundary conditions. 
To this aim we use the probabilistic approach by studying the quantitative homogenization of
backward stochastic differential equations (SDEs) associated with slow-fast systems of reflected SDEs.

\vskip 0,2truein



\end{abstract}

 \maketitle

\section{Introduction}\label{S:1}

Homogenization (or averaging)  is a mathematical technique used to rigorously derive an effective (homogenized) model that approximates the behavior of a system with rapidly oscillating, periodic coefficients. This is typically done in the limit as the oscillation scale $\eps \to 0$, where $\eps>0$ is a small parameter 
  representing the ratio between the microscopic and macroscopic scales.
 Here is a prototype of questions in the study of   homogenization for partial differential equations (PDEs) with periodic coefficients.  
Suppose that $A(x)=(a^{ij} (x))_{1\leq i, j\leq n}$ is a multivariate 1-periodic, $n\times n$ matrix-valued function on $\R^d$ that is bounded and uniformly elliptic. Let 
$$
 \cL f(x) := {\rm div} (A(x)\nabla f(x) ) = \sum_{i,j=1}^n \frac{\partial  }{\partial x^i}\left( a^{ij} (x) \frac{\partial f(x)}{\partial x^j} \right)
 $$
 be defined in the distributional sense, and 
 $\cL^{\eps} f(x) =  {\rm div} (A(x/\eps )\nabla f(x)) $ for $\eps >0$. 
Let $D\subset \R^n$ be a smooth domain and
consider the following PDE  on $D$ with suitable boundary conditions:   
\begin{equation}\label{pde}
  \left\{
    \begin{array}{lllll} 
\cL^{\eps, t, x} u_\eps(x) = - f(x) \quad \text{in } D, \\  
u_\eps(x)~~ \text{ satisfies some boundary conditions on } \partial D.
 \end{array} \right.
\end{equation}
Equation \eqref{pde} models a stationary process in an inhomogeneous material with rapidly oscillating microstructure, and $\eps > 0$ is a parameter measuring the inhomogeneity scale. It has been established in mathematics that under some suitable conditions on 
the periodic diffusion matrix $A$ and suitable boundary conditions, the solution $u_\eps$ of \eqref{pde} converges in a suitable sense  to some function $u$ on $\overline D$  that solves a homogenized PDE: 
\begin{equation}\label{pde0}
  \left\{
    \begin{array}{lllll} 
\widetilde{\cL}u = - f \quad \text{in } D, \\ 
u ~~ \text{ satisfies  some boundary conditions on } \partial D,
\end{array} \right.  
\end{equation}
where $\widetilde \cL$ is a second order elliptic differential operator having constant coefficients. 
 The homogenization result stated above shows that the inhomogeneous materials with rapidly oscillating microstructure may be approximately described via an effective homogeneous material.  There exist quite a few methods in studying homogenization problems, and among them we would like to mention asymptotic expansions, correctors, energy method. We refer the interested reader to the 
 classical book \cite{BLP} for a detailed accounts of these methods.

On the other hand, it is well known (see, e.g., \cite{CZ95}) 
that there is a continuous symmetric Feller process $X=(X_t, t\geq 0; \mP_x, x\in \R^d)$ on $\R^d$
having strong Feller property having $\cL$ as its infinitesimal generator.
For $\eps >0$, define $X^\eps_t = \eps X_{\eps^{-2}t}$. Then $X^\eps$ is a continuous symmetric Feller 
process on $\R^d$ having $\cL^\eps$ as its infinitesimal generator. One can run $X^\eps$ to solve 
PDE \eqref{pde}. For example,  the solution to the Dirichlet boundary value problem 
\begin{equation}\label{pde-D}
  \left\{
    \begin{array}{lllll} 
\cL^\eps u_\eps(x) = - f(x) \quad \text{in } D, \\  
u_\eps(x) = \varphi (x) \quad  \text{  on } \partial D.
 \end{array} \right.
\end{equation}
is given by  
\begin{equation}\label{e:pr}
u_\eps (x)= \mE_x \left[ \varphi(X^\eps_{\tau^\eps_D}) \right] + \mE_x \int_0^{\tau^\eps_D} f (X^\eps_s) ds, 
\quad x\in D,
\end{equation} 
where $\tau^\eps_D :=\inf\{t> 0:X^\eps_t \notin D\}$; see \cite{CZ95}. 
 A natural question arises: 
  does $X^\eps(t)$ converge   in the distributional sense as $\eps\to 0$? if so, what is the limit process?
 The problem was solved by Bensoussan, J.L. Lions and Papanicolaou in the classical book \cite{BLP} for  
 case when $A(x)$ is periodic in $\R^n$. The weak convergence of the diffusion processes $X^\eps$
provides a way to study homogenization problem for PDEs \eqref{pde} by using  probabilistic method via \eqref{e:pr}.

Tanaka \cite{Ta84} later studied periodic homogenization  for  diffusions with boundary conditions in half-spaces.
 He obtained the limit homogenized process
  under some
 smooth and periodic assumptions on the coefficients for  the diffusion process $X_t^\varepsilon$ on the half-space $D=\{x\in\mR^n, x_1 >0\}$ with generator
$$\cL_\varepsilon =\sum_{i,j=1}^n a^{ij}(\varepsilon^{-1}x)\frac{\partial^2}{\partial x^i\partial x^j} +{\varepsilon}^{-1}\sum_{i=1}^n b^{i}(\varepsilon^{-1}x)\frac{\partial }{\partial x^i}+\sum_{i=1}^n c^i(\varepsilon^{-1}x)\frac{\partial }{\partial x^i}
$$
{ subject to}
 the boundary condition:
$$
\Gamma_\varepsilon u (x):=\sum_{i,j=2}^n\alpha^{ij}(\varepsilon^{-1}x)\frac{\partial^2u (x) }{\partial x^i\partial x^j} +{\varepsilon}^{-1}\sum_{i= 2}^n\beta^i(\varepsilon^{-1}x)\frac{\partial u (x) }{\partial x^i} +\sum_{i=1}^n\gamma^i(\varepsilon^{-1}x)\frac{\partial u (x) }{\partial x^i}=0
 \quad \text{ on }\partial D.
$$

 The probabilistic representation can be extended to solutions of nonlinear PDEs as well. For example, consider the following semilinear PDE:
\begin{equation}\label{semipde0}
 \left\{
    \begin{array}{lllll}
\frac{\partial u^{\varepsilon}}{\partial t}(t, x)=\sum_{i,j=1}^n a^{i j}(\eps^{-1}x) \frac{\partial^{2} u^{\varepsilon}}{\partial x^i \partial x^j}(t, x)+\eps^{-1}\sum_{i=1}^n b^{i}(\eps^{-1}x)\frac{\partial u^{\varepsilon}}{\partial x^i}(t, x)+f\left(\eps^{-1}x, u^{\varepsilon}(t, x)\right),\\

u^{\varepsilon}(0, x)=g(x).
\end{array}\right.
\end{equation} 
For every $\eps>0$, assuming the above PDE has a solution $u^\eps$, then the solution can be represented by the solution of a backward SDE. {The research on} backward SDEs 
and their connection to non-linear PDEs has been pioneered by Pardoux and Peng \cite{PP90}, \cite{PP92}. Due to this connection, the homogenization problem of this type of PDEs can also be studied via the probabilistic representation and
 the homogenization of related backward SDEs; 
see,  for example, \cite{par99}, \cite{par-ver97}. 

In many applications, one often  encounters  asymptotic problems for systems having multi-scales.
 In this regards,  consider the deterministic system in $\mR^n$:
\be\label{ode}
d \tilde {X}_t^\varepsilon = \varepsilon b( \tilde X_t^\varepsilon, Y_t)dt, \quad \tilde X_0^\varepsilon = x,
\ee
where $ Y_t$ is {an $\mR^n$-valued function of $t\in [0, \infty)$,}
 $\varepsilon >0$ is a small numerical parameter.
 If $b$ is bounded, then $\tilde X^\eps_t$ moves at a speed no larger than $\eps \| b\|_\infty$.
 In this case,  $\tilde X^\eps_t$ converges to  the constant function $X_t\equiv x$  uniformly on every finite time interval $[0, T]$ as $\varepsilon \to 0$.
 The behavior of $\tilde X_t^\varepsilon$ on time intervals of order $\varepsilon^{-1}$ or of larger order is usually of {primary} interest. 
Thus we set  $X_t^\varepsilon := \tilde X_{t/\varepsilon}^\varepsilon$. 
Then $X^\eps$  satisfies
\begin{equation}\label{ode2}
 dX_t^\varepsilon = b(X_t^\varepsilon,  Y_{t/\eps} )dt  \quad \hbox{with } X_0^\varepsilon = x.  
  \end{equation}  
  The function $X^\eps$ is called a slow evolution as compared with the ``fast" motion $t\to  Y_{t/\eps}$.  
  The study of this system \eqref{ode2} on a finite time interval is equivalent to the study of
system \eqref{ode} on time intervals of order $\varepsilon^{-1}$. {Suppose that}  ${Y}$ is periodic, $b(x, y)$ is bounded, continuous, and satisfy
a Lipschitz condition in $x$ with the {Lipschitz} constant independent of $y$. Then the limit 
\begin{equation}\label{odebarb}
 \bar{b}(x) := \lim_{T \to \infty} \frac{1}{T} \int_0^T b(x,  {Y}_s) \, ds   
\end{equation}
exists for each $x \in \mR^n$. 
If the convergence in \eqref{odebarb} is uniform in $x$, then one can show that the trajectory $ { X^\varepsilon}$ converges to $\bar{x}_t$, uniformly on every finite time interval as $\varepsilon \to 0$, where $\bar{x}$ is determined by the equation:
$$
d{\bar{x}}_t = \bar{b}(\bar{x}_t)dt, \quad \bar{x}_0 = x.
$$

{The property that   the trajectory $X^\varepsilon$ of \eqref{ode2} converges to $\bar{x}$ as $\eps \to 0$
 is called the averaging principle.} {Note that for  $Y^\eps_t := Y_{t/\eps}$, then $dY^\eps_t = \eps^{-1}\dot  Y_{t/\eps} dt$.
   If the dynamic system \eqref{ode2} is subject to random noises, one has the following form of random dynamic system: 
\be \label{e:1.9}
\left\{
    \begin{array}{lllll}
      d X_t^{\eps}=\sigma(X_t^{\eps},Y_t^{\eps})d W_t+b(X_t^{\eps},Y_t^{\eps})d t \hskip 0.9truein \hbox{with } X_0^{\eps}=x\in\mR^n, \\
      d Y_t^{\eps}=\e^{-1/2}\sigma^{(2)}(X_t^{\eps},Y_t^{\eps})dB_t+\e^{-1}b^{(2)}(X_t^{\eps},Y_t^{\eps})d t
      \quad \hbox{with  } Y_0^{\eps}=y\in\mR^n,
     \end{array}
  \right.
\ee
where $B$ and $W$ are two independent standard Brownian motions on $\R^n$, 
{ and $\sigma (x, y)$ and $\sigma^{(2)} (x, y)$ are two $n\times n$ matrices on $\R^n\times \R^n$.}
The interest is to study the limit behavior of the slow process $X^\eps$ as $\eps \to 0$. 
The solution processes $(X^\eps, Y^\eps)$ of the above  random dynamic system can be used to solve the  the following second order parabolic PDEs;
that is, for a bounded continuous function $f(x, y)$ defined on $\R^n \times \R^n$, $u^\eps (t, x, y):= \mE^{(x, y)} [ f(X^\eps_t, Y^\eps_t)]$ solves 
\be \label{e:1.10} 
\left\{
    \begin{array}{lllll}
\frac{\partial  }{\partial t} u^\eps = \big( \cL_1+ \eps^{-1} \cL_2\big)    u^\eps,\\
\\
 u^\eps(0,x,y)=f(x),
 \end{array}
  \right.
  \ee
where  for $(x, y)\in \R^n\times \R^n$, 
\begin{eqnarray*}
\cL_1 \varphi (x, y) &:=&  \sum_{i, j=1}^n a^{ij}(x, y) \frac{\partial^2 \varphi (x, y)}{\partial x^i \partial x^j} + \sum_{i=1}^n b^i(x, y) \frac{ \partial \varphi (x, y)}{\partial x^i}, \\
\cL_2 \varphi (x, y)  & :=&  \sum_{i, j=1}^n a^{(2),ij}(x, y) \frac{\partial^2 \varphi (x, y)}{\partial y^i \partial y^j}
 + \sum_{i=1}^n b^{(2),i}(x, y) \frac{ \partial \varphi (x, y)}{\partial y^i}.
\end{eqnarray*}
Here $(a^{ij}(x, y))_{1\leq i, j\leq n}= \frac12\sigma (x, y) \sigma (x, y)^*$ and $(a^{(2),ij}(x, y))_{1\leq i, j\leq n}= \frac12\sigma^{(2)} (x, y) \sigma^{(2)} (x, y)^*$,
{ where $\sigma^*$  denote the adjoint matrix of  $\sigma$.}
In this way the homogenization for the stochastic system \eqref{e:1.9} can be used to study the homogenization of the PDE \eqref{e:1.10}. 
Averaging principle of multi-scale stochastic systems has  attracted much attention in recent decades. To investigate the averaging problem for SDEs and SPDEs, there are several effective methods, including the time discretization method, Poisson equation (also called the corrector's method),
 asymptotic expansion method, and there are fruitful results in this respect; see, for example, 
\cite{BC}, \cite{ce09},  \cite{cf09}, \cite{dsz18}, \cite{fl11}-\cite{HLL1}, \cite{KY04}, \cite{RSX}, \cite{R-X211}, \cite{R-X21}, \cite{WY20}, \cite{Xu}.
There exists now a large amount of literature on homogenization results concerning the asymptotic problem of SDEs and PDEs from both the probabilistic and the analytic point of views.  Among the pioneering work in this area are \cite{h68}, \cite{Ko78}, \cite{PV81}, \cite{PV82}. 

\medskip

 Recently, the authors \cite{C-W22}  studied 
qualitative averaging principle for stochastic variational inequalities and applied it to
get the qualitative homogenization result for multi-scale semilinear second order partial differential equations PDEs with nonlinear Neumann conditions. 
The purpose of this paper to study quantitative homogenization  for two classes of multi-scale  semilinear second order partial differential equations on convex domains with Neumann boundary condition. 
We use probabilistic approach by first studying the quantitative homogenization of 
backward SDEs associated with  slow-fast systems of reflected SDEs.  {For the  multi-scale  semilinear second order partial differential equations on smooth 
  bounded domains with nonlinear Neumann boundary condition, we apply a class of generalized backward SDEs coupled with forward reflected SDEs and establishes a strong convergence rate  for the homogenization. For second order semilinear PDEs on unbounded convex domains with zero Neumann boundary conditions, we obtain a homogenization convergence rate $\epsilon^{1/4}$.} These quantitative homogenization results are new  both  
in probability theory and in PDE.

 \medskip
  
  The rest of the paper is organized as follows.
 Section \ref{S:2} presents a quick survey  on averaging principles of multi-scale forward stochastic dynamical systems on Euclidean spaces and 
their connection to the homogenization of systems of semilinear parabolic PDEs. 
Section \ref{S:3} reviews some homogenization results for multi-scale reflected SDE system on  nonempty closed convex domains  
 in $\R^n$.  
 Quantitative homogenization results of
backward SDEs associated with   slow-fast systems of reflected SDEs on convex domains are studied in Section \ref{S:4}.
These results are used to get the quantitative  homogenization result for viscosity solutions  of
  multi-scale  semilinear second order partial differential equations on smooth 
  bounded convex domains with nonlinear Neumann boundary condition in Section \ref{S:4.2},
   and on general unbounded convex domains with zero Neumann boundary condition in Section \ref{S:4.3}. 
 
 \medskip
 
  In this paper, we use := as a way of definition.
For a subset  $D\subset \R^n$, 
 we denote by $\cC([0,+\infty); D)$ the space of continuous functions 
  defined on $[0,\infty)$ taking values in $D$,
 and by  $\mW^{2,p}(D)$ the usual Sobolev space   $D$ of order $(2, p)$ with $p\geq 1$.

\section{Averaging principle of multiscale SDEs and homogenization of related PDEs} \label{S:2}

\subsection{Averaging principle of multiscale SDEs}   

In this section we discuss the asymptotic behavior of stochastic systems involving multi-time-scale diffusion processes.  
For every $x, y\in \R^n$, 
we consider  the following  stochastic system with both slow process $X^{\eps, x, y}$ and fast process
 $Y^{\eps, x, y}$ as in \eqref{e:1.9}: 
\begin{equation}\label{eq_fast_slow_components}
  \left\{
    \begin{array}{lllll}
      d X_t^{\eps;x,y}=\sigma(X_t^{\eps;x,y},Y_t^{\eps;x,y})d W_t+b(X_t^{\eps;x,y},Y_t^{\eps;x,y})d t
      \quad \hbox{with } X_0^{\eps;x,y}=x\in\mR^n, \\
\\
      d Y_t^{\eps;x,y}=\e^{-1/2}\sigma^{(2)}(X_t^{\eps;x,y},Y_t^{\eps;x,y})dB_t+\e^{-1}b^{(2)}(X_t^{\eps;x,y},Y_t^{\eps;x,y})d t  
       \quad \hbox{with }  Y_0^{\eps;x,y}=y\in\mR^n.
     \end{array}
  \right.
\end{equation}
Here $b, ~b^{(2)}:\mR^n\times\mR^n\rightarrow\mR^n, ~\sigma, ~\sigma^{(2)}:\mR^n\times\mR^n\rightarrow \mR^{n\times {m}}$
 are measurable functions, $(W_t)_{t\geq0}$ and $(B_t)_{t\geq0}$   are two independent $n$-dimensional standard Brownian motions on a filtered probability space $(\Omega, \sF, \mP,
 \{\cF_t\}_{t\geq0} )$. 
Such multi-scale models  have been used to describe many dynamical systems in which not all components
evolve at the same rate. For example,  they are used  to describe stochastic volatility in finance (\cite{FFK12}), climate weather interactions (see \cite{kifer1}, \cite{kifer2}),  intracellular biochemical reactions (\cite{KK13}, \cite{kuehn}), and so on. 

The generator corresponding to SDE \eqref{eq_fast_slow_components} is $\cL_1+\eps^{-1} \cL_2$ on $\R^n\times \R^n$, where
for $(x, y)\in \R^n\times \R^n$, 
\be\label{LL2}
&&\cL_1:= \sum_{i, j=1}^n a^{ij}(x,y)\frac{\partial^2}{\partial x^i\partial x^j} +\sum_{i=1}^n b^i(x,y)\frac{\partial}{\partial x^i},\nonumber\\
&&\cL_2:= \sum_{i, j=1}^n a^{(2), ij}(x,y)\frac{\partial^2}{\partial y^i\partial y^j} +\sum_{i=1}^n b^{(2),i}(x,y)\frac{\partial}{\partial y^i}.
\ee
Here similar to {\eqref{e:1.10}}, $a(x, y)=\frac12\sigma \sigma^* (x, y)$ and $a^{(2)}(x, y)=\frac12\sigma^{(2)}  {\sigma^{(2)}}^* (x, y)$.

 To handle  the different time scales and the cross interactions between the fast and
slow variables  in   \eqref{eq_fast_slow_components}, in general, 
 for each $x, y\in \R^n$ we need to consider an auxiliary process  $Y^{x, y}_t$, 
  which is the solution of the following  SDE:
\be\label{Yx}
dY^{x, y}_t = b^{(2)}(x, Y^{x, y}_t)\mathrm{d}t + \sigma^{(2)}(x, Y^{x, y}_t)\mathrm{d}B_t
\quad \hbox{with } Y^{x, y}_0=y.
\ee

The following theorem gives an averaging principle in the weak convergence sense. It
 can be proved by applying the time discretization method combined with the martingale problem. 
 We refer the readers to \cite[Section 4.3]{Freidlin}, \cite[Chapter 7]{fw84} and \cite{par-ver01}, \cite{par-ver03} for its  proof. 
 
\bt\label{averaging1}
Assume the coefficients in \eqref{eq_fast_slow_components} are twice continuously differentiable and bounded, and 
 there exists  $\bar{b} (x) $ and $\bar{a} (x)=(\bar{a}^{ij} (x))$ so  that
 for any $\delta>0$ {and $x, y\in \R^n$, 
\be\label{avecff} 
&&\lim_{T\to\infty}\mP\left(\left|\frac{1}{T} \int_{t}^{t+T} b(x, Y_s^{x, y}) \, ds - \bar{b}(x)\right|>\delta\right)=0,\nonumber\\
&&\lim_{T\to\infty}\mP\left(\left|\frac{1}{T} \int_{t}^{t+T} a^{ij}(x, Y_s^{x, y}) \, ds - \bar{a}^{ij}(x)\right|>\delta\right)=0
\quad\hbox{for every } 1\leq i, j\leq n. 
\ee
Then the process $X^{\eps;x,y}$ converges weakly  in $\cC([0,\infty);\mR^n)$ as $\eps \to 0$ to the diffusion process $\bar{X}^x$, 
 which is the unique solution of the SDE independent of $y\in \R^n$:
 }
\be\label{barx}
d\bar{X}_t=  \bar{\sigma}(\bar{X}_t)dW_t + \bar{b}(\bar{X}_t)dt, \quad \bar{X}_0=x.
\ee
If $\sigma(x,y)=\sigma(x)$ does not depend on the fast variable, then $X^{\eps;x,y}$ converge in probability 
 in $\cC([0,\infty);\mR^n)$ as $\eps \to 0$ to the process $\bar{X}^x$.
\et

The above averaging result combined with the probabilistic representation as in \eqref{e:pr} yields the following homogenization of the following Cauchy problem
 for  $u^\eps (t, x, y)$ on $[0, \infty)\times \R^n \times \R^n$:
 \be\label{CP} \begin{cases}
\frac{\partial  }{\partial t} u^\eps =(\cL_1+ \eps^{-1} \cL_2)u^\eps,\\
u^\eps(0,x,y)=f(x).
\end{cases}
\ee

\bt\label{HMCP} {\rm (\cite[Section 4.3]{Freidlin})}
Under the assumptions of Theorem \ref{averaging1}, for every bounded continuous $f$ on $\R^n$,
the solution $u^\eps (t, x, y)$ of \eqref{CP} converges as $\eps \to 0$ pointwise on $[0, \infty)\times \R^n \times \R^n$
to the solution ${u(t, x)}$ of the following   Cauchy problem:
\ce\begin{cases}
\frac{\partial u}{\partial t}=\bar{\cL}u,\\
 u(0,x)=f(x),\end{cases}
\de
where  $\bar{\cL} = \sum_{i, j=1}^n \bar a^{ij} (x) \frac{\partial^2}{\partial x^i \partial x^j} + \sum_{i=1}^n \bar b^i (x)  \frac{\partial}{\partial x^i } $.
\et

 The conditions on the coefficients in Theorems \ref{averaging1} and \ref{HMCP} can be relaxed. For example, suppose that in \eqref{eq_fast_slow_components}, $a^{(2)}(x,y)$ is non-degenerate in $y$ uniformly with respect to $x$, and there exist constants $\lambda\in\R$, $p>2$ and $C, C_1, C_2>0$, for every $x\in\R^n$, $b^{(2)}(x,\cdot), ~\sigma^{(2)}(x,\cdot)$ satisfy that
\ce
&&2\langle y-y', b^{(2)}(x,y)-b^{(2)}(x,y')\rangle+\|\sigma^{(2)}(x,y)-\sigma^{(2)}(x,y')\|^2\leq \lambda|y-y'|^2(1\vee\log|y-y'|^{-1}), \\
&&2\langle y, b^{(2)}(x,y)\rangle+\|\sigma^{(2)}(x,y)\|^2\leq -C_1|y|^p+C_2, ~~{ \|\sigma^{(2)}(x,y)\|^2\leq C(1+|y|)}.
\de 
According to \cite{ZXC}, then 
for each $x\in \R^n$, the diffusion process $Y^{x, y}$ of \eqref{Yx} is exponentially ergodic and
 has a unique  invariant probability measure $\pi^x$ { that is independent of $y\in \R^n$}. 
 Consequently, the condition \eqref{avecff} of  Theorem \ref{averaging1} is satisfied with 
 \begin{equation}\label{e:4.7a}
 \bar{a}(x):=\int_{\mathbb{R}^{n}} a(x,y)  \pi^x (\mathrm{d}y)  
\quad \hbox{and} \quad 
\bar{b}(x):=\int_{\mathbb{R}^{n}} b(x,y)  \pi^x (\mathrm{d}y).
\end{equation}
Assume in addition that $b^{(2)}, ~\sigma^{(2)}$ are Lipschitz continuous in $x$ uniformly in $y$, and $b, ~\sigma$  are Lipschitz continuous in both $x$ and $y$. Then by an argument from \cite[Section 4.3]{Freidlin} and \cite[Chapter 7]{fw84}, 
 the conclusions of Theorems \ref{averaging1} and \ref{HMCP} still hold.

Recently  R\"ockner and Xie \cite{R-X21} considered  strong convergence of averaging principle for a multi-scale system which is more {general} than \eqref{eq_fast_slow_components}
with less regular coefficients. They used  the Zvonkin transformation combined with the Poisson (or, corrector's) equation method. In particular, 
they obtained an optimal strong convergence rate of the homogenization as a corollary of the main result in \cite{R-X21}, which we state below.

\bt\label{strongaveraging} {\rm (\cite{R-X21})}
 Suppose $a^{(2)}(x,y)$ is non-degenerate in $y$ uniformly with respect to $x$, and  
$$
\lim_{|y|\rightarrow\infty} \sup_{x\in \R^n}  \langle y, b^{(2)}(x, y) \rangle = -\infty.
$$
Suppose also that $a(x,y)=a(x)$ (that is, $a$ does not depend on the fast variable) and  is non-degenerate in $x$, 
 $b$, $b^{(2)}$,  $\sigma^{(2)}$ are both bounded, $\theta$-H\"older continuous in $x$, $\delta$ -H\"older continuous in $y$ for $ \theta, \delta > 0$, and $\sigma\in\cC_{b}^{1}$.
 Then for every $x, y\in \R^n$, 
 any $T >0$ and every  $p\geq1$,
$$
\sup_{t\in[0,T]} \mathbb{E} \left[ |X_{t}^{\eps, x, y} - \bar{X}^x_{t}|^{p}  \right] \leqslant C_{T}\eps^{(\theta\wedge1)p/2},
$$
where $\bar X^x$ is the diffusion given by \eqref{barx}. 
\et

\subsection{Homogenization of systems of semilinear parabolic PDEs}

Backward SDEs decoupled with 
 multiscale forward SDEs of the form like
 \eqref{eq_fast_slow_components} is also an effective   probabilistic machinery for the study of homogenization problems of some  semilinear parabolic PDEs.  
 Consider the following system of semilinear PDEs  on Euclidean spaces:
  for $k = 1, \ldots, m$,
\be\label{pde-bsde}
\begin{cases}
\frac{\partial u_k^\varepsilon}{\partial t} + \cL_1 u_k^\varepsilon + \eps^{-1} \cL_2 u_k^\varepsilon + f_k \left( x, y, u^\varepsilon, \left( \partial_x u^\varepsilon \sigma(x, y), \varepsilon^{-1/2}\partial_y u^\varepsilon {\sigma^{(2)}(x, y)} \right) \right) = 0,\\
u_k^\varepsilon (T, x, y) = h_k(x).
\end{cases}
\ee
Here $\cL_1, ~\cL_2$ are the differential operators given in \eqref{LL2}, and $f=(f_1,\ldots, f_m): \mR^n\times\mR^n\times \mR^m\times {(\mR^{m\times  n}, \mR^{m\times  n})}\to \mR^m$ and $h=(h_1,\ldots, h_m): \mR^n\to\mR^m$ are Borel measurable functions.
 For each
 $x, y\in \R^n$, $\eps >0$ and $0\leq t<T$, 
consider the following backward SDE: 
\be\label{bsde-pde}
{U^{\eps; t, x,y}_s}=h(X^{\eps;t,x,y}_T)+\int_s^T f(X_r^{\eps;t,x,y}, Y_r^{\eps;t,x,y}, U^{\eps; t, x,y}_r , Z_r^{\eps; t, x,y} ) dr - \int_s^T 
Z_r^{\eps; t, x,y} d\tilde{W}_r,
\quad  s\in[t,T],
\ee
Here $\tilde{W}_s=(W_s, B_s)^\top$, where $W$ and $B$ are two independent $n$-dimensional standard Brownian motions on a complete probability space $(\Omega, \sF, \mP)$,
and $(X^{\eps;t,x,y}, Y^{\eps;t,x,y})$ is a solution of the following multiscale SDE system {for $s\in [t, T]$}, 
 \be\label{e:4.9}
 \left\{
    \begin{array}{lllll}
      d X_s^{\eps;t,x,y}=\sigma(X_s^{\eps;t, x,y},Y_s^{\eps;t,x,y})d W_s+b(X_s^{\eps;t, x,y},Y_s^{\eps;t,x,y})d s , \\
      d Y_s^{\eps;t,x,y}=\e^{-1/2}\sigma^{(2)}(X_s^{\eps;t,x,y},Y_s^{\eps;t,x,y})dB_s+\e^{-1}b^{(2)}(X_s^{\eps;t,x,y},Y_s^{\eps;t,x,y})d s ,
           \end{array}
  \right.
\ee
with $X_t^{\eps;t,x,y}=x\in\mR^n$ and $Y_t^{\eps;t,x,y}=y\in\mR^n$. 
When all the coefficients $b, b^{(2)}, \sigma, \sigma^{(2)}, ~f$ and $h$ are  bounded and {$C^\infty$-smooth with bounded derivatives}.  
It is well known that  for every $\eps >0$, $x, y\in \R^n$ and $0\leq t<T$, 
 \eqref{bsde-pde} and \eqref{e:4.9} admit a unique strong solution $\{(X^{\eps;t,x,y}_s, Y^{\eps;t,x,y}_s, U^{\eps; t,x,y}_s, Z^{\eps; t,x,y}_s); 
s\in [t, T]\}$
with the minimum augmented filtration $\{\cF^{t, T}_s\}_{s\in [t, T]}$ generated by $\{\tilde W_s -\tilde W_t; s\in [t, T]\}$.
So $U^{\eps, t, x,y}_t$ is deterministic. 
 In addition, it is established in  \cite{PP90}, \cite{PP92} that
\begin{equation}\label{e:4.10}
u^{\eps}(t, x, y):=U^{\eps; t, x,y}_t
\end{equation}
 gives the unique classical solution to 
 the semi-linear equation \eqref{pde-bsde}.

Suppose in addition that $a$ and $a^{(2)}$  are uniformly elliptic, and $b, b^{(2)}, \sigma$ and $ \sigma^{(2)}$ are multivariate periodic in $y$-variable with period 1. Then for each $x\in \R^n$, the unique strong solution  $Y^x$ to 
 $$ d Y_s^{x}= \sigma^{(2)} (x,Y_s^x)dB_s+ b^{(2)}(x,Y_s^x)d s, \quad s\geq 0,
 $$
  can be viewed as a diffusion process  which is ergodic and has a unique invariant probability measure $\pi^x$ on $\mT^n$. 
Define $\bar a$ and $\bar b$ as in \eqref{e:4.7a}, and for $k=1,\cdots, m$, 
$$
\bar{f}_k (x,u, z):=\int_{\mT^n} f_k (x,y, u, (z\sigma(x,y),0))\pi^x(dy)
\quad \hbox{for } (x, u, z) \in \mR^n \times \mR^{m}\times \mR^{m\times n}. 
$$
Let 
$$
\bar{\cL} := \sum_{i, j=1}^n \bar a^{ij} (x) \frac{\partial^2}{\partial x^i \partial x^j} + \sum_{i=1}^n \bar b^i (x)  \frac{\partial}{\partial x^i } 
\quad \hbox{on } \R^n.
$$
Using the probabilistic representation \eqref{e:4.10},
 Briand and Hu  \cite{BH99}  obtained
   the following  homogenization result for \eqref{pde-bsde}.

\bt\label{pde-bsde3} {\rm (\cite[Theorem 2.4]{BH99})}
Suppose that $b, b^{(2)}, \sigma, \sigma^{(2)}, f$ and $h$ are { $C^\infty_b$-smooth and multivariate periodic in $y$-variable with period $1$}, 
and  $a$ and $a^{(2)}$  are uniformly elliptic. Let $T>0$ and $u$ be the classical solution of the following PDE system:  
\begin{align}\label{pde-bsde31}
\begin{cases}
\frac{\partial u_k}{\partial t}+\bar{\cL}u_k+\bar{f}_k(x, u,\nabla_x u)=0,\\
u_k(T,x)=h_k(x),
\end{cases}
\quad \hbox{for } ~k=1,\cdots, m.
\end{align}
and $u^\eps$ be the classical solution  of  \eqref{pde-bsde}. Then 
 for every $(t,x,y)\in[0,T]\times\mR^n\times\mR^n$, 
$$
\lim_{ \eps \to 0 } u_k^{\eps}(t,x,y) =  u_k(t,x)
\quad \hbox{for } k=1,\cdots, m.
$$
\et
{In a very recent work, Sheng et al \cite{SWYZ} have relaxed the conditions on the coefficients to Lipschitz continuity and extended the above homogenization result to the viscosity solution for \eqref{pde-bsde} and obtained a strong convergence rate.}

\section{Averaging principle of multiscale SDEs reflected in convex domains}  \label{S:3}

In this section  we review the  averaging principle and the rate of strong convergence recently obtained by the authors in \cite{C-W22}
for  multi-scale  fully coupled reflected stochastic systems on  closed convex domains in $\R^n$.
The results obtained there are in a more general framework of stochastic variational inequalities.
To present them in a transparent way,  in this section we will  take the 
lower-semicontinuous convex functions $\varphi$ and $\psi$ used in the formulation of 
stochastic variational inequalities in \cite{C-W22} to be the indicator functions of closed convex domains in $\R^n$
to reduce them to the  more conventional setting of reflected SDEs. 
Let $\bar D_1$ and $\bar D_2$ be two nonempty closed convex domains   in $\R^n$.
Consider the following multi-scale fully coupled reflected SDE system for $(X^{\eps, x, y}, Y^{\eps, x, y})$
taking values in $\bar D_1 \times\bar D_2$:
\begin{equation}\label{RSDE1}
  \left\{
    \begin{array}{lllll}
      d X_t^{\eps;x,y}=\sigma(X_t^{\eps;x,y},Y_t^{\eps;x,y})d W_t+b(X_t^{\eps;x,y},Y_t^{\eps;x,y})d t
      { + } 
      \mathbf{n} (X_t^{\eps;x,y}) d\phi_t^{\eps;x,y},
          \medskip   \\
      d Y_t^{\eps;x,y}=\eps^{-1/2} \sigma^{(2)}(X_t^{\eps;x,y},Y_t^{\eps;x,y})dB_t+ \eps^{-1} b^{(2)}(X_t^{\eps;x,y},Y_t^{\eps;x,y})d t
      { + } 
      \mathbf{n}^{(2) (Y_t^{\eps;x,y})}d\phi_t^{(2),\eps;x,y}, 
         \end{array}
  \right.
\end{equation}
with $X_0^{\eps;x,y}=x\in\bar{D}_1$ and $Y_0^{\eps;x,y}=y\in\bar{D}_2$, 
  where $W$ and $B$ are two independent $n$-dimensional standard Brownian motion on a complete probability space $(\Omega, \cF, \mP)$ with the filtration  $\{\cF_t\}_{t\geq0}$ generated by $\{\tilde{W}_t, t\geq0\}$, where  $\tilde{W}_t:=\left(W_t, B_t\right)^\top$.
  Here  $X_t^{\eps;x,y} \in \bar D_1$ and $Y_t^{\eps;x,y}\in \bar D_2$ { for all $t\geq 0$,}
  $\mathbf{n}(x)$ and $\mathbf{n}^{(2)}(y)$ are  the unit inward normal vectors at $x\in\partial D_1$ and at $y\in\partial D_2$, respectively,
   $\phi_t^{\eps;x,y}$ and $\phi_t^{(2),\eps;x,y}$ are real-valued continuous increasing processes 
with initial value $0$ that increase only when $X_t^{\eps;x,y} \in \partial D_1$ and $Y_t^{\eps;x,y}\in \partial D_2$,  respectively. 

We consider the following conditions. 
\begin{enumerate}
\item[\textbf{(A.1)}]   There exists a constant $c_1>0$ such that for any $x, ~y\in\mR^d$,
\be\label{linear}
  |b(x,y)|+\|\sigma(x,y)\|\leq c_1(1+|x|+|y|).
\ee
    For any $x, ~x'\in\mR^d$ and any $~y\in\mR^d$, there exists a constant $C_1>0$  such that
    \ce
    &&|b(x,y)-b(x',y)| {+}
    \|\sigma(x,y)-\sigma(x',y)\|\leq C_1|x-x'|.\\
    &&|b^{(2)}(x,y)-b^{(2)}(x',y)|+\|\sigma^{(2)}(x,y)-\sigma^{(2)}(x',y)\|\leq C_1|x-x'|.
   \de

\item[\textbf{(A.2)}]  $b^{(2)}, ~\sigma^{(2)}$ are bounded and there exist constants $\lambda>0$, and $C_2>0$
\ce
&&|b(x,y)-b(x,y')| {+}
    \|\sigma(x,y)-\sigma(x,y')\|\leq C_2|y-y'|.\\
&&\<y_1-y_2, b^{(2)}(x,y)-b^{(2)}(x,y')\>+\|\sigma^{(2)}(x,y)-\sigma^{(2)}(x,y')\|^2\leq -\lambda|y-y'|^2.
\de

\item
[\textbf{(A.3)}] ~~
  $a^{(2)}(x,y)$ is nondegenerate in $y$ uniformly in $x$, and there exist constants $\lambda\in\R$, $p>2$ and $C_1, C_2, C_3>0$ such that for all $x\in\R^n$, \ce
2\langle y, b^{(2)}(x,y)\rangle+\|\sigma^{(2)}(x,y)\|^2\leq -C_1|y|^p+C_2,\quad 
\|\sigma^{(2)}(x,y)\|\leq C_3(1+|y|).
\de 
\end{enumerate}

 The following result on the existence and uniqueness of the solution to the coupled reflected SDE system \eqref{RSDE1}  
as well as its averaging principle is established  in  \cite[Theorem 3.2]{C-W22} as its special case.

\begin{theorem}\label{HM_RSDE_W}
 Suppose that  \textbf{(A.1)}-\textbf{(A.2)} and the following \textbf{(A.4)} hold: 

\begin{enumerate}
\item[\rm \textbf{(A.4)}]
  There are functions $\bar a^{ij}(x)$ and $\bar b^i(x)$ for $1\leq i, j\leq n$ so that
 for any $T>0$, there is a constant $\beta (T)>0$ with $ \lim_{T\rightarrow\infty}\beta(T)=0$
so that
  \ce
      &&\Big| \frac{1}{T} \mE\int_t^{t+T}b^i(x,Y_s^{1;x,y})d s-\bar{b}^i(x) \Big| \leq\beta(T)(1+|x|+|y|),\\
 && \Big| \frac{1}{T}\mE\int_t^{t+T}a^{ij}(x,Y_s^{1;x,y})d s-\bar{a}^{ij}(x) \Big|\leq\beta(T)(1+|x|+|y|),
\de
where $a(x,y)=\frac12\left(\sigma \sigma^*\right)(x, y)$, and $Y^{1;x,y}$ is given by:
\be\label{aveYequa}
 d Y_t^{1;x,y}= \sigma^{(2)}(x,Y_t^{1; x,y})dB_t+ b^{(2)}(x,Y_t^{1;x,y})d t
 { + } 
  \mathbf{n}(Y_t^{1; x,y}) d\phi_t^{1;x,y}, \quad Y_0^{1;x,y}=y\in\bar{D}_2.
\ee
\end{enumerate}
Then for every $\eps>0$ and $x, y\in \R^n$,
  \eqref{RSDE1} has a unique strong solution. Moreover, the sequence $\{X_t^{\eps;x,y}\}$ converges 
  weakly in $C([0, \infty); \R^n) $ as $\eps\downarrow0$ to the process $\bar{X}_t^x$ given by
    \begin{equation}\label{eq_average_overline_X_t_x}
  d\bar{X}_t^x=\bar{\sigma}(\bar{X}_t^x)d W_t+\bar{b}(\bar{X}_t^x)d t 
  { + } 
  \mathbf{n}(\bar{X}_t)d\bar{\phi}_t
  \quad \hbox{with } \bar{X}_0^x=x\in\bar{D}_1.
\end{equation}
\end{theorem}

Strong convergence  may not hold in general and we refer the readers to \cite[Section 4.1]{liudi} for  a
counter-example
 in the context of a couple SDE system on the whole Euclidean space 
 (that is, when $D_1=D_2=\R^n$).

 For each $x\in \bar D_1$ and $y\in \bar D_2$,  we 
 next consider the case where the diffusion coefficient of the
 normally reflected diffusion of the slow equation does not depend on the fast component:
\be \label{e:4.1} 
 \begin{cases}
 dX_t^{\e, x, y}=  \sigma(X_t^{\e, x, y})dW_t + b(X_t^{\e, x, y},Y_t^{\e, x, y})dt
  +   \mathbf{n}(X_t^{\e, x, y})d\phi_t^{\eps;x,y}, \\
 dY_t^{\e, x, y}= \e^{-{1/2}}\sigma^{(2)}(X_t^{\e, x, y},Y_t^{\e, x, y})dB_t +  \e^{-1}b^{(2)}(X_t^{\e, x, y},Y_t^{\e, x, y})dt 
  +  \mathbf{n}^{(2)}(Y_t^{\e, x, y})d\phi_t^{(2),\eps;x,y} , 
 \end{cases}
\ee
with $X_0^{\eps;x,y}=x\in\bar{D}_1$ and   $Y_0^{\eps;x,y}=y\in\bar{D}_2.$

\medskip

 By the arguments  for the proof of  \cite[Theorem 3.1]{rwz},  we know
 $Y^{1;x,y}_t$ is a homogeneous Markov process
that  is ergodic and  has a unique invariant  probability measure $\pi^{x}$. Moreover, 
there exist constants $C>0$ and $\gamma >0$ such that for any $x\in\bar{D}$ and $u\in \mR$, 
\be\label{lefef}
	|{\mathbb{E}}b(x, Y_t^{1;x,y}) - \bar{b}(x)|
	\leq C\big(1 +|x|+|y|\big)e^{-\gamma t/2},
\ee
where $\bar{b}(x) := {\int_{\bar{D}_2}b(x,y)\pi^{x}(dy)}$. 
Hence the condition {\bf (A.4')} on p. 178 of \cite{C-W22} is satisfied.
Thus as a corollary of  \cite[Theorem 4.4]{C-W22}, we have following quantitative   averaging principle
for the coupled reflected SDE system \eqref{e:4.1}.

\bt \label{HM4} 
Suppose  that \textbf{(A.1)-(A.3)} hold. 
Then for every $0\leq t\leq T$,
{every $x\in \bar D_1$ and $y\in \bar D_2$,}
 $$
 \mE \Big[ | {X}_t^{\e, x, y}-\bar{X}^x_t |^2 \Big] \leq C\e^{1/2},
 $$
 where {$(\bar X^x, \bar \phi^x )$ }
 is the reflected diffusion on $\bar D_1$ given by 
\begin{equation*} 
  d\bar{X}^x_t= \sigma ( \bar{X}^x_t) d W_t+ \bar{b}(\bar{X}^x_t) d t
+
  \mathbf{n}(\bar{X}^x_t)d \bar{\phi}^x_t
   \quad \hbox{with } \bar{X}^x_0=x \in \bar D_1.
\end{equation*}
\et

 \section{Homogenization of BSDEs and PDEs with Neumann boundary conditions}\label{S:4}

It is well known that there is a close connection between second order elliptic operators 
and diffusion processes; see \cite{Freidlin}, \cite{ik-wa} and the references therein. 
 For instance,  consider the following partial differential equation with Neumann boundary condition in 
 an open smooth domain $D$  in $\R^n$: 
\be \label{e:4.1a}
\begin{cases}
\frac{\partial u}{\p t} = \cL u(t,x)- g(t,x)u(t,x), \quad \hbox{for} ~t\geq0 \hbox{ and } x\in D,\\
u(0,x)=h(x),  \quad x\in D,\\
\frac{\partial u}{\p \mathbf{n}}=0, \quad  x\in\p  D,
\end{cases}
\ee
where $\cL$  has the following form:
 \be\label{cL}
 \cL:=\sum_{i, j=1}^na^{ij}(x)\frac{\partial^2}{\partial x^i\partial x^j} +\sum_{i=1}^n b^i(x)\frac{\partial}{\partial x^i},
 \ee
 and ${\bf n}$ is the unit inward normal vector field on $\partial D$. 

 {Under some suitable conditions, \eqref{e:4.1a} has a a unique solution and it is given by
 (see, for example, \cite{Freidlin}, \cite{fw84})
\be\label{FK1}
u(t,x)=\mE\left[h(X^{0,x}_t)\exp{\left\{-\int_0^t g(t-s, X^{0,x}_s)ds\right\}}\right],
\quad t\geq 0, x\in \bar D.
\ee
 Such connection has been extended to non-linear second order elliptic equations via backward SDEs, 
pioneered by  Pardoux and Peng \cite{PP90}, \cite{PP92}.}

 For $t\geq 0$ and $x\in \bar D$, let $X^{t,x}$  be the normally reflected diffusion process on $\bar D$ given by
 \be\label{rsde0}
dX_s^{t,x}=\sigma(X_s^{t,x})dW_s+b(X_s^{t,x})dt + \mathbf{n}(X_s^{t,x})d\phi_s^{t,x} \quad \hbox{for } s\geq t
\ \hbox{ with } \  
X_t^{t,x}=x\in \bar{D},
\ee
where ${\frac12}\sigma (x) \sigma^* (x)= a(x)$ and 
$s\mapsto \phi_s^{t,x}$ is a continuous non-decreasing real-valued process with $\phi_t^{t,x}=0$ that increases
only when $X_s^{t,x} \in \partial D$. We call $\phi^{t,x}$ the local time process for the reflected diffusion $X^{t,x}$. 
 For the reflected diffusion process $X^{t,x}$ on $\bar D$  of  \eqref{rsde0}, 
Pardoux and Zhang \cite{par-zh} studied  
 the following type of  backward SDE involving the  local time process $\phi^{t, x}$ for $X^{t,x}$:
\be\label{bsde02}
U_s^{t,x} = h(X_T^{t,x}) + \int_s^T f(r, X_r^{t,x}, { U}_r^{t,x}, Z_r^{t,x}) dr + \int_s^T g(r,X_r^{t,x}, { U}_r^{t,x}) d\phi^{t,x}_r- \int_s^T Z_r^{t,x} dW_r
\ee
for $s\in [t, T]$. 

\medskip

When $b, ~\sigma$ are Lipschitz continuous, and $f, ~g, ~h$ satisfy the following condition:
\begin{enumerate}
\item [(\textbf{A})]
There exist constants $C_i>0, ~i=1,2,3, ~\alpha_1, ~\alpha_2\in\mathbb{R}$ so that for any $x, ~x'\in\mR^n$,  $ u, ~u'\in\mR$, $z, ~z'\in\mR^n$,
\ce\begin{aligned}
&\<u-u',g(t,x,u)-g(t,x,u')\>\leq -\alpha_1|u-u'|^2;  \\
&\<u-u',f(t,x,u,z)-f(t,x,u',z)\>\leq -\alpha_2|u-u'|^2; \\
&|f(t,x,u,z)-f(t,x,u,z')|\leq C_1|z-z'|;   \\
&|f(t,x,u,z)|\vee |g(t,x,u)|\leq C_2(1+|x|+|u|+|z|);     \\
&|h(x)-h(x')|\leq C_3|x-x'|,
\end{aligned}
\de
\end{enumerate}
Pardoux and Zhang \cite{par-zh}  established  the existence and uniqueness of solution of \eqref{bsde02}, 
 which is a pair of $\mathbb{R} \times \mathbb{R}^n$-valued progressively measurable process
 $\{( U_s^{t,x}, Z_s^{t,x} ); s\in [t, T]\}$
 with respect to the minimum augmented 
filtration generated by $\{W_r-W_t; r\in[t,s]\}_{t\leq s\leq T}$. Note that $U_t^{t,x}$ is deterministic.  
 Set 
$$
u(t,x) :=U_t^{t,x} \quad \hbox{for }  (t,x) \in [0, T] \times \bar{D}.
$$
 For the situations when $D$ is a smooth, open and  bounded smooth domain and when $D$ is a convex non-empty domain,
  it is established in  \cite[Proposition 4.1 and Theorem 4.3]{par-zh} and in \cite[Theorem 3.1 and Section 5.1]{par-ras}, respectively,  
  that $u(t,x)$  is continuous on $[0, T] \times \bar{D}$   
and  is a viscosity solution of the following semi-linear PDE with Neumann boundary condition:
\be\label{N-pde}
\begin{aligned}\begin{cases}
&\frac{\partial u}{\partial t}(t,x) + \cL u(t,x) + f(t,x,u(t,x),( (\nabla_x u) \sigma)(t,x)) = 0
 \quad  \hbox{for } (t,x)\in(0, T)\times D; \\
&  \frac{\partial u}{\partial \mathbf{n}}(t,x) = g(t,x,u(t,x))  
 \quad \hbox{for } (t,x)\in(0, T)\times\partial D; \\
&u(T,x) = h(x),\ \hbox{for } x \in \bar{D};
\end{cases}
\end{aligned}
\ee
Here $\nabla_x u $ denotes the gradient for the function  $u$ taken in the $x$-variable { and $\cL$ is  the operator given in \eqref{cL}.
This connection between the forward-backward SDEs with reflection \eqref{rsde0} and \eqref{bsde02}
and the semi-linear PDE \eqref{N-pde} with Neumann boundary condition provides a path to connect homogenization problems for \eqref{N-pde}
with  the homogenization problem for forward-backward SDEs with reflection. 
Indeed, this has been explored in a recent paper  \cite{C-W22} of the authors in the framework of Neumann-type semilinear parabolic 
variational inequalities. To avoid technicality,   we present this result in \S  \ref{S:4.1}  for the  case where the lower-semicontinuous 
convex function $\psi$ in \cite[(1.3)]{C-W22} is the indicator function of a convex domain $D_1$ in $\R^n$ so its sub-differential $\partial \psi (x) $
at $x\in \partial D_2$ is just the unit inward normal vector field at $x$ for $D_2$. 
We then present the main results of this paper on the quantitative homogenization for  semi-linear PDE \eqref{N-pde} with Neumann boundary condition,
first on smooth, bounded and convex domains in \S \ref{S:4.2} and then on non-smooth convex and possibly unbounded domains in \S \ref{S:4.3}.

\subsection{Homogenization of BSDEs and PDEs with Neumann boundary conditions} \label{S:4.1}

Let $D_1$ be an open convex and bounded domain  $D_1$ in $\R^n$ given by $D_1=\{x\in\R^n; l(x) \geq 0\}$ for some convex function
 $l\in\cC_b^2(\R^n)$ with $|\nabla l(x)|=1$ on $\partial D_1$, and $D_2$ is another  non-empty convex domain of $\R^n$.
 We use ${\bf n}$ and  ${\bf n}^{(2)}$ to denote  the unit inward normal vector field on $\partial D_1$ and $\partial D_2$, respectively. 
 The notations  $\frac{\partial_x  }{\partial \mathbf{n}} $ and $\frac{\partial_y }{\partial \mathbf{n}^{(2)}}$ stand  for the normal derivatives
 taken in the $x$-variable on $\partial D_1$ and in the $y$-variable on $\partial D_2$, respectively.

 In this subsection, we consider    the following multi-scale parabolic semi-linear   PDE on $D_1\times D_2$ with Neumann condition:
 \be\label{pdeN1}
\left\{
  \begin{array}{lllll}
    \frac{\partial }{\partial t} u^{\e} +(\cL_1+ \e^{-1} \cL_2)u^{\e}+f(x,y,u^{\e})  =0
     \quad  \hbox{for }  (t,x,y)\in [0,T]\times D_1\times{D}_2 , \medskip\\
      \frac{\partial_x  }{\partial \mathbf{n}} u^{\e} = g(x,u^{\e})  \quad  \hbox{for }   (t,x,y)\in [0,T]\times \partial D_1\times  \bar{D}_2,\medskip\\
 \frac{\partial_y }{\partial \mathbf{n}^{(2)}}u^{\e}=0 \quad  \hbox{for }    (t,x,y)\in [0,T]\times \bar{D}_1\times  \partial {D}_2,
        \medskip\\
       u^{\e}(T,x,y)=h(x) \quad  \hbox{for } (x, y) \in \bar{D}_1 \times \bar D_2,
  \end{array}
\right.
\ee
where $\cL_1, ~\cL_2$ are given in \eqref{LL2},  
   the coefficients $f, ~g, ~h$ are continuous functions satisfying  the following {condition
\textbf{(H.3)}.
\begin{enumerate}
\item[\textbf{(H.3)}]
There exist constants $C_i>0, ~i=0,1,2,3, ~\alpha\in\mathbb{R}$ so that for any  $ u, ~u'\in\mR$, $x, x', y, y'\in\mR^n$,
\ce\begin{aligned}
&\<u-u',g(t,x,u)-g(t,x,u')\>\leq \alpha|u-u'|^2;  \\
&|f(t,x,y,u)-f(t,x,y, u')|\leq C_0|u-u'|^2; \\
&|f(t,x,y,u)-f(t,x',y',u)|\leq C_1(|x-x'|+|y-y'|);   \\
&|g(t,x,u)|\leq C_2(1+|x|+|u|);  ~~|h(x)-h(x')|\leq C_3|x-x'|,
\end{aligned}
\de
\end{enumerate}}

As mentioned above, this type of PDE has a connection  with the following  reflected forward-backward SDE: 
for { $t\in [0, T)$}  and $(x, y)\in  \bar{D}_1 \times \bar{D}_2$,
\be\label{bsdee0} 
U^{\e,t, x,y}_s&=&h(X^{\e,t, x,y}_T)+\int_s^Tf( X^{\e,t, x,y}_r, Y^{\e,t, x,y}_r,U^{\e,t, x,y}_r)dr\nonumber\\
&&\skip 0.6truein +\int_s^Tg(X^{\e,t, x,y}_r, U^{\e,t, x,y}_r)d\phi^{\e,t, x,y}_r-\int_s^TZ^{\e,t, x,y}_rd W_r,
\quad  s\in [t, T], 
\ee 
where  $X^{\eps, x, y}$ and $Y^{\eps, x, y}$ are  coupled reflected diffusion processes on $\bar D_1$ and $\bar D_2$ having 
$X_t^{\eps, t,x,y}=x$ and $Y_t^{\eps, t,x,y} =y$ { determined by} 
\begin{eqnarray}\label{ebsde00}\left\{
  \begin{array}{lll}
    d X^{\eps, x, y}_s=b(X_s^{\eps, t, x, y},Y_s^{\eps, t, x, y}) d s+\sigma(X_s^{\eps, t, x, y},Y^{\eps, t, x, y}_s)d W_s
  +\mathbf{n}(X^{\eps, t, x, y}_s)d \phi^{\e, t,x,y}_s
     ,  \medskip  \\
     d Y^{\eps, t, x, y}_s=\e^{-1}b^{(2)}(X_s^{\eps, t, x, y},Y_s^{\eps, t, x, y})d s+\e^{-1/2}{\sigma}^{(2)}(X_s^{\eps, t, x, y},Y_s^{\eps, t, x, y})d B_s
   +\mathbf{n}^{(2)}(Y^{\eps}_s) d  \phi^{(2),\eps,t,x,y}_s.
       \end{array}
\right.
\end{eqnarray}

Denote $\tilde{W}_\cdot :=(W_\cdot, B_\cdot)^\top$. Suppose conditions  in Theorem \ref{HM_RSDE_W} and \textbf{(H.3)} hold,  then by \cite[Theorem 1.7] {par-zh} and  \cite[Proposition 2.5]{par-ras},  for each $\e>0$, \eqref{bsdee0} has a unique solution $(U^{\e,t, x,y},Z^{\e,t, x,y})$, 
which is $\mathbb{R} \times \mathbb{R}^n$-valued progressively measurable with respect to the minimum augmented 
filtration generated by $\{\tilde{W}_r-\tilde{W}_t; r\in[t,s]\}_{t\leq s\leq T}$ and satisfies that 
\ce
\mE \left[ \sup_{s\in[t,T]}|U^{\e,t, x,y}_s|^2 \right] +\mE\int_t^T|Z^{\e,t, x,y}_s|^2ds<\infty.
\de
 Note that  $U^{\e,t, x,y}_t$ is deterministic. Define 
$$
u^{\e}(t,x,y):=U^{\e,t, x,y}_t.
$$ 
Then $u^{\e}(t,x,y)$ is continuous on $[0,T]\times\bar{D}_1\times \bar{D}_2$  (\cite[Theorem 3.1]{par-ras}). 
Moreover, by arguments in the proofs of \cite[Theorem 4.3]{par-zh} and \cite[Theorem 4]{za},
$u^{\e}(t,x,y)$ is a viscosity solution to \eqref{pdeN1}.

Due to this connection, the homogenization of \eqref{pdeN1} can be studied via 
the homogenization of   \eqref{ebsde00}
in the $S$-topology, to be  stated in Theorem \ref{HMbsde} below. The $S$-topology 
is introduced in \cite{Jak} on the Skorohod space of c\`adl\`ag (right continuous with left limit at every point) functions.

 For  each $x\in \bar D_1$ and $y\in\bar{D}_2$,   let $Y^{1;t,x, y}$  be  the reflected diffusion on $\bar D_2$ determined by 
\be\label{frozen}
 d Y^{1,t, x, y}_s= b^{(2)}(x,Y_s^{1, t, x, y})d s+ {\sigma}^{(2)}(x,Y_s^{1, t,  x, y})d B_s
     +\mathbf{n}^{(2)}(Y^{1,t, x, y}_s)d\phi^{(2), t, x, y}_s  \quad\hbox{for } s\geq t ,
\ee
with   $Y_t^{1, t, x, y}=y\in \bar{D}_2$.

By the arguments in \cite[Theorem 3.1]{rwz}, under the conditions in Theorem \ref{HM4},  {$Y^{1, t, x, y}$}
 has a unique  invariant measure $\pi^x$.
  Set $\bar{f}(x,u) := {\int_{\bar{D}_2}f(x,z,u)\pi^{x}(dz)}$. Then
there exist constants $C>0$ and $\gamma >0$ such that for any $x\in\bar{D}_1, ~u\in \mR$ and $s\in [t,T]$, 
\be\label{lefef2}
	|{\mathbb{E}}f(x,Y_s^{1,t,x,y},u) - \bar{f}(x,u)|
	\leq C\big(1 +|x|+|y|+|u|\big)e^{-\gamma  (s-t)/ 2}.
\ee
Since  $\bar{f} (x, u)$ is  Lipschitz continuous on $\bar D_1 \times \R$,}
 by \cite[Theorem 4.3.1]{ZJF} again,  the following backward SDE has a unique solution
 {$(U^{t, x}, Z^{t, x})$:}
 \be\label{bsde2b} 
U^{t, x}_s =h(X^{t, x}_T) + \int_s^T \bar{f}(X^{t, x}_r, U^{t, x}_r) \, dr+\int_s^Tg(X^{t,x}_r,U^{t,x}_r)d\phi^{t,x}_r - \int_s^T Z^{t, x}_r \, dW_r
\quad \hbox{for }  s\in [t,T].
\ee
Here {$X^{t, x}$} is given by the reflected SDE as in \eqref{eq_average_overline_X_t_x}:
\ce
dX^{t, x}_s = \bar{b}(X^{t, x}_s) \, ds + \bar{\sigma}(X^{t, x}_s) \, dW_s+\mathbf{n}(X^{t, x}_s) d\phi^{t, x}_s \quad \hbox{with } X^{t, x}_t = x.
\de

 Note that 
 $(U^{t,x}_s, Z^{t, x}_s)_{t\leq s\leq T}$
  are  progressively measurable with respect to the natural  
filtration generated by $\{{W}_r-{W}_t; r\in[t,s]\}_{t\leq s\leq T}$ and  $U^{t, x}_t$ is deterministic
for every $t\in [0, T]$.

 The following theorem can be proved following the arguments { for}  \cite[Theorem 5.2]{C-W22}.

\begin{theorem}\label{HMbsde}
Suppose that the  conditions of Theorem \ref{HM_RSDE_W} and  \textbf{(H.3)} hold.
Then for every $0\leq t<T$, 
as $\eps \to 0$, $U^{\e,t, x,y}$ converges  in distribution 
to $\bar{U}^{t, x}$  in the  $S$-topology in the space $\D([t,T];\mR)$ 
of right continuous functions with left limits. In addition,  for any $(t,x,y)\in[0,T]\times\bar D_1\times
 \bar{D}_2$, 
 $$
 \lim_{\eps \to 0} u^{\e}(t,x,y) = \bar{u}(t,x), 
 $$
 where  $\bar{u}(t,x):=\bar{U}_t^{t, x}$  is the unique viscosity solution of the following averaged  Neumann type PDE:
\be \label{e:5.7}
\left\{
  \begin{array}{lllll}
    \frac{\partial  }{\partial t} \bar{u} + \bar{\cL}\bar{u}+\bar{f}(x,\bar{u}) =0 
       \quad  \hbox{for }  (t,x)\in[0,T]\times {D_1},  \medskip \\
     \frac{\partial }{\partial \mathbf{n}} \bar{u} = g(x,\bar{u}) \quad  \hbox{for }  (t,x)\in[0,T]\times\partial D_1,  \medskip
    \\
    \bar{u}(T,x)={h}(x) \quad \hbox{for }   x\in \bar{D}_1,
  \end{array}
\right.
\ee
Here $\bar{\cL} :=\sum_{i,j=1}^n\bar{a}^{ij}(x)\frac{\partial ^2}{\partial x^i\partial x^j}
+ \sum_{i=1}^n  \bar{b}^i(x)\frac{\partial}{\partial x^i}$. 
\et

\subsection{Quantitative homogenization of PDEs with nonlinear Neumann boundary conditions  on smooth bounded domains} \label{S:4.2}

 Theorem \ref{HMbsde} gives 
  the homogenization result for $u^{\eps}(t,x,y)$. A natural but important question  is that if we can obtain the convergence rate.  
  In this section, we will address  this question by utilizing 
   the strong convergence rate established { in \cite[Theorem 4.4]{C-W22}, as stated in} Theorem \ref{HM4}. 
Suppose $D_1$ and $D_2$ are $n$-dimensional  non-empty open bounded convex domains satisfying $D_i:=\{x\in\R^n; l_i(x) > 0\}$
  for some convex functions
 $l_i\in\cC_b^2(\R^n)$ with $|\nabla l_i(x)|=1$ on $\partial D_i$ with $i=1, 2.$
Consider the following PDE 
with nonlinear Neumann boundary condition: 
\be\label{pdeN20}
\left\{
  \begin{array}{lllll}
    \frac{\partial  }{\partial t}  u^{\e} + (\cL_1+ \e^{-1} \cL_2)u^{\e}+f(x,y,u^{\e}) = 0
     \quad \hbox{for } (t,x,y)\in [0,T]\times \bar D_1\times\bar{D}_2, \medskip\\
        \frac{\partial_x }{\partial \mathbf{n}}  u^{\e} =g(x,u)     \quad  \hbox{for }   (t,x,y)\in [0,T]\times \partial D_1 \times\bar{D}_2,\medskip\\ 
       \frac{\partial_y }{\partial \mathbf{n}^{(2)}} u^{\e} =0. \quad   \hbox{for }  (t,x,y)\in [0,T]\times \bar{D}_1 \times\partial {D}_2,  \medskip\\
       u^{\e}(T,x,y)=h(x) \quad  \hbox{for }  ( x, y) \in \bar{D}_1 \times \bar D_2,
  \end{array}
\right.
\ee
where $\cL_1, ~\cL_2$ are   the operators  given in \eqref{LL2} with $a(x,y)=a(x)=\frac12\sigma\sigma^*(x)$.
Consider the following generalized backward SDE:  for $s\in [t, T],$
\begin{eqnarray}\label{bsde2}
U_s^{\varepsilon,t, x, y} &=& h(X_T^{\varepsilon,t, x, y}) + \int_s^T f(X_r^{\varepsilon,t, x, y}, Y_r^{\varepsilon,t, x, y}, U_r^{\varepsilon,t, x, y}) \, dr + \int_s^T g(X_r^{\varepsilon,t, x, y},  U_r^{\varepsilon,t, x, y}) \, d\phi^{\eps, t, x, y}_r \nonumber\\
&&  - \int_s^T Z_r^{\varepsilon,t, x, y} \, dW_r,  
\end{eqnarray}
which is  decoupled with slow-fast reflected SDE system on convex domains 
  $D_1$ and $D_2$:
\be\label{rsde-bsde2} 
\begin{cases}
dX_s^{\eps, t, x, y} = b(X_s^{\eps, t, x, y}, Y_s^{\eps, t, x, y}) \, ds + \sigma(X_s^{\eps, t, x, y}) \, dW_s 
+\mathbf{n}(X_s^{\eps, t, x, y} )d\phi^{\eps, t, x, y}_s,  \\
dY_s^{\eps, t, x, y} = \eps^{-1}b^{(2)}(X_s^{\eps, t, x, y}, Y_s^{\eps, t, x, y} ) \, ds 
+ \eps^{-1/2} \sigma^{(2)}(X_s^{\eps, t, x, y}, Y_s^{\eps, t, x, y}) \, dB_s  
+ \mathbf{n}^{(2)}(Y_s^{\eps, t, x, y}) d \phi_s^{(2), \eps, t, x, y},  \end{cases}
\ee
{with $ X_t^{\eps, t, x, y}  = x\in \bar{D}_1$ and $Y_t^{\eps, t, x, y}  = y\in\bar{D}_2$, } 
where $B$ and $W$ are two independent Brownian motions.

According to \cite[Theorem 1.7] {par-zh} and  \cite[Proposition 2.5]{par-ras},  for each $\e>0$, \eqref{bsde2} has a unique solution $(U^{\e,t, x,y},Z^{\e,t, x,y})$, 
which is $\mathbb{R} \times \mathbb{R}^n$-valued progressively measurable with respect to the minimum augmented 
filtration generated by $\{\tilde{W}_r-\tilde{W}_t; r\in[t,s]\}_{t\leq s\leq T}$. 
 Note that  $U^{\e,t, x,y}_t$ is deterministic. Then
$$
u^{\e}(t,x,y):=U^{\e,t, x,y}_t
$$
 is continuous on $[0,T]\times\bar{D}_1\times \bar{D}_2$  (\cite[Theorem 3.1]{par-ras}). 
Moreover, by arguments in the proofs of \cite[Theorem 4.3]{par-zh} and \cite[Theorem 4]{za},
$u^{\e}(t,x,y)$ is a viscosity solution to \eqref{pdeN20}.

Due to this connection, the homogenization of \eqref{pdeN20} can be studied via 
the homogenization of  \eqref{bsde2}. 
The averaged generalized backward SDE has the following form: 
\be\label{bsdeav2} 
{U}^{t, x}_s={h}({X}^{t, x}_T)+\int_s^T\bar{f}({X}^{t, x}_r, {U}^{t, x}_r)dr+\int_s^Tg({X}^{t, x}_r, {U}^{t,x}_r)d{\phi}^{t, x}_r
-\int_s^T{Z}^{t, x}_rd W_r,  \quad s\in[t,T],
\ee
where $\bar{f}$ satisfies \eqref{lefef2}, and ${X}^{t,x}$ is given by the following equation:
\be\label{ave_rsde}
d{X}^{t,x}_s=\bar{b}({X}^{t,x}_s)ds+{\sigma}({X}^{t,x}_s)dW_s+\mathbf{n}({X}^{t,x}_s)d{\phi}^{t,x}_s, \ s\geq t,
\quad \bar{X}^{t,x}_t=x\in\bar{D}_1,
\ee
with $\bar{b}$ being  the same as in Theorem \ref{HM4}.

By \cite[Theorem 1.7] {par-zh} and  \cite[Proposition 2.5]{par-ras}}, BSDE  \eqref{bsdeav2} has 
 a unique solution  $({U}^{t,x},{Z}^{t,x})$ and ${U}^{t,x}_t$ is   deterministic for every $t\in [0, T]$.

The following is the first main result of the paper, which gives a quantitative homogenization result for BSDE \eqref{bsde2}.

\bt\label{HMbsde21}
Suppose the conditions of  Theorem \ref{HM4} and \textbf{(H.3)} hold. Then for every $s\in[t,T]$, 
$x\in \bar D_1$ and $y\in \bar D_2$,  for any $p\in(1,2)$, 
\be\label{hmbrate}
\mathbb{E} \left[ |U^{\varepsilon,t, x, y}_s-U^{t, x}_s|^2 \right] \leq C\e^{\frac{p-1}{2(2p-1)}}.
\ee
 \et

 Set $\bar{u}(t,x):=U^{t,x}_t$, where $U^{t,x}$ is the unique solution of \eqref{bsde2b}. Then $\bar{u} (t, x)$ is continuous 
 { on $[0, T]\times \bar D_1$} and, by \cite[Theorem 4.3]{par-zh},  is a viscosity solution to the following PDE:
\ce
\left\{
  \begin{array}{lllll}
     \frac{\partial }{\partial t}  \bar{u} + \bar{\cL}\bar{u}+\bar{f}(x,\bar{u})=0  \quad  \hbox{for }  (t,x)\in[0,T]\times D_1,  \medskip \\
       \frac{\partial  }{\partial \mathbf{n}} \bar{u} =g(x,\bar{u})    \quad  \hbox{for }   (t,x)\in[0,T]\times\partial D_1,  \medskip
    \\
    \bar{u}(T,x)={h}(x)  \quad  \hbox{for }  x\in\bar{D}_1,
  \end{array}
\right.
\de
where $\bar \cL :=\sum_{i,j=1}^n{a}^{ij}(x)\frac{\partial ^2}{\partial x^i\partial x^j}
+ \sum_{i=1}^n  \bar{b}^i(x)\frac{\partial}{\partial x^i}$.
  Hence as a direct application of Theorem \ref{HMbsde21}, 
 we have the following quantitative homogenization result for 
 the viscosity solution $u^{\eps}(t,x,y)$ of equation \eqref{pdeN20}.
  
\begin{theorem}\label{HM-pde-N3}
Suppose the conditions of Theorem \ref{HMbsde21} hold. Then for every $T>0$,  there is a constant $C>0$ so that for any $p\in(1,2)$,
$$|u^{\e}(t,x,y)-\bar{u}(t,x)|\leq C\e^{\frac{p-1}{4(2p-1)}}
\quad   \hbox{for every } (t, x, y) \in[0,T]\times \bar{D}_1\times \bar{D}_2.
$$
\end{theorem}

\medskip
 
To prove Theorem \ref{HMbsde21}, we first need to establish some  estimates for  the solutions of \eqref{bsde2} and  \eqref{rsde-bsde2}.

\begin{proposition}\label{MME1} 
Eq. \eqref{rsde-bsde2}  and  Eq. \eqref{ave_rsde} 
admit a unique strong solution$(X^{\eps, t, x, y},  \phi^{\eps, t, x, y}, Y^{\eps, t, x, y},  \phi^{(2),\eps, t, x, y})$ and $(X^{t, x}, \phi^{t,x})$ respectively,
 satisfying that for any $q\geq1$,
\begin{eqnarray}\label{MMXY1}
&\sup _{\varepsilon} \mE \Big[ \sup_{s\in[t, T]}\left|X_s^{\eps, t, x, y}\right|^{2 q} \Big] 
\leqslant C_{q, T}\left(1+|x|^{2 q}+|y|^{2q}\right).\nonumber\\
&\sup _{\varepsilon} \sup_{s \in[t, T]} \mE \Big[ \left|Y_s^{\eps, t, x, y}\right|^{2q} \Big]
 \leqslant C_{q, T} \left(1+|x|^{2q}+|y|^{2q}\right),\nonumber\\
&\mE \Big[ \sup_{s\in[t, T]}\left|X_s^{t, x}\right|^{2q} \Big] \leq C_{q, T}\left(1+|x|^{2 q}\right).
\end{eqnarray}
Furthermore,  for any $\eta>0$,  there exists a constant $C_{T,\eta}>0$ independent of $\eps$ so that
\begin{eqnarray}\label{MMphi1}
&\sup _{\varepsilon} \mE \left[ e^{\eta \phi^{\eps, t, x, y}_T} \right] \leq C_{T,\eta}<\infty,
\quad \mE \left[ e^{\eta \phi^{t, x}_T} \right] \leq C_{T,\eta}<\infty.
\end{eqnarray}
\end{proposition}

\begin{proof}
 Existence and uniqueness of the strong solution of   \eqref{rsde-bsde2}  and   \eqref{ave_rsde} are established
in    \cite[Theorem 6.1]{C-W22} and in \cite{LS}, respectively.
   
 The first estimate in \eqref{MMXY1} can be proved by arguments similar to \cite[Proposition 3.3]{C-W22} and the other two estimates 
   can be established in a similar way as that for 
 (3.4) and (4.5) in \cite{C-W22}. 

Note that  $l_1\in\mathcal{C}_b^2$ and  $\mathbf{n}(x)= \nabla l_1(x)$ for $x\in\partial D_1$. We get by It\^o's formula,
$$
\begin{aligned}
&l_1(X_s^{\varepsilon,t,x,y})-l_1(x)\\
=&\int_t^s\left \langle \nabla l_1(X_r^{\varepsilon,t,x,y}),b(X_r^{\varepsilon,t,x,y},Y_r^{\varepsilon,t,x,y})\right\rangle dr
+\int_t^s\left\langle\nabla l_1(X_r^{\varepsilon,t,x,y}),\sigma(X_r^{\varepsilon,t,x,y})\right\rangle dW_r\\
&+\frac12\int_t^s\mathrm{tr}(\nabla^2l_1\sigma\sigma^{\top})(X_r^{\varepsilon,t,x,y})dr
{ - } \int_t^s|\nabla l_1(X_r^{\varepsilon,t,x,y})|^2d\phi_r^{\varepsilon,t,x,y}.
\end{aligned}
$$
It  follows that for every $s\in[t,T]$,
$$
\begin{aligned}
\phi_s^{\eps, t, x, y}&=l_1\left(X_s^{\varepsilon,t,x,y}\right)-l_1\left(x\right)  -\int_t^s \langle \nabla l_1\left(X_r^{\varepsilon,t,x,y}\right), b\left(X_r^{\varepsilon,t,x,y}, Y_r^{\varepsilon,t,x,y}\right)\rangle d r \\
& -\frac{1}{2} \int_t^s \mathrm{tr}(\nabla^2l_1\sigma\sigma^{\top})\left(X_r^{\varepsilon,t,x,y}\right) d r 
 -\int_t^s \langle \nabla l_1\left(X_r^{\varepsilon,t,x,y}\right), \sigma\left(X_r^{\varepsilon,t,x,y}\right) d W_r\rangle.
\end{aligned}
$$
By the boundedness of $b, ~\sigma$  on $\bar D_1 \times\bar D_2$, and $\bar D_1$ respectively, we have
$$
\begin{aligned}
\mathbb{E}
e^{\eta \phi_T^{\eps, t, x, y}}&\leq \mathbb{E}\left[e^{\eta  l_1(X_T^{\eps, t, x, y})-\eta l_1(x)-\eta\int_t^T\langle \nabla l_1\left(X_r^{\varepsilon,t,x,y}\right), b\left(X_r^{\varepsilon,t,x,y}, Y_r^{\varepsilon,t,x,y}\right)\rangle d r-\frac{\eta}{2} \int_t^T \mathrm{tr}(\nabla^2l_1\sigma\sigma^{\top})\left(X_r^{\varepsilon,t,x,y}\right) d r}e^{-\eta\int_t^T \langle \nabla l_1\left(X_r^{\varepsilon,t,x,y}\right), \sigma\left(X_r^{\varepsilon,t,x,y}\right) d W_r\rangle}\right]\\
&\leq C_{T,\eta}+ C_T\mathbb{E}\left[e^{-2\eta\int_t^T \langle \nabla l_1\left(X_r^{\varepsilon,t,x,y}\right), \sigma\left(X_r^{\varepsilon,t,x,y}\right) d W_r\rangle}\right]\\
&\leq C_{T,\eta}.
\end{aligned}
$$
 This gives  the first estimate in \eqref{MMphi1}.

Similarly we can prove for any $q\geq1$,
\ce
\mE \Big[ \sup_{s\in[t, T]}\left|X_s^{t, x}\right|^{2q} \Big] \leq C_{q, T}\left(1+|x|^{2 q}\right)
\quad \hbox{and} \quad 
\mE \Big[ e^{\eta \phi^{t, x}_T}\Big]\leq C_{\eta, T} <\infty.
\de
\end{proof}

We shall also need the following estimate concerning the local times $\phi^{t,x}, ~\phi^{\varepsilon, t,x,y}$ of \eqref{ave_rsde} and \eqref{rsde-bsde2}.
\begin{proposition}\label{phiE1}  Divide $[t,T]$ into subintervals with the same length $\Lambda>0$.  For any $q>1$,
\be \label{e:4.21} 
\mE\left[\max_k\left(\phi^{t,x}_{(k+1)\Lambda}
-\phi^{t,x}_{k\Lambda}\right)^{2q}\right]\leq C\Lambda^{q-1}, 
\quad \sup_{\varepsilon >0}\mE\left[\max_k\left(\phi^{\varepsilon,t,x,y}_{(k+1)\Lambda}
-\phi^{\varepsilon,t,x,y}_{k\Lambda}\right)^{2q}\right]\leq C\Lambda^{q-1}.
\ee
\end{proposition}

\begin{proof}
By applying \cite[Lemma 2.6]{Ta79}, for the local time $\phi^{t,x}$ of \eqref{ave_rsde}, for any $t\leq s<u\leq T$,
\begin{align*}
\phi^{t,x}_u-\phi^{t,x}_s\leq C\sup_{s\leq r_1<r_2\leq u}\left|\int_{r_1}^{r_2}\bar{b}(X^{t,x}_l)dl+\int_{r_1}^{r_2}\sigma(X^{t,x}_l)dW_l\right|.
\end{align*}
Thus by applying Proposition \ref{MME1},
\begin{align*}
&\mE\left[\max_k \left(\phi^{t,x}_{(k+1)\Lambda}
-\phi^{t,x}_{k\Lambda}\right)^{2q}\right]\\
\leq&C\mE\left[\max_k\left|\int_{k\Lambda}^{(k+1)\Lambda}\bar{b}(X^{t,x}_r)dr\right|^{2q}\right]
+C\mE\left[\max_k\sup_{s\in[k\Lambda,(k+1)\Lambda]}\left|\int_{k\Lambda}^{s}\sigma(X^{t,x}_r)dW_r\right|^{2q}\right]\\
\leq&C\Lambda^{2q}+\sum_kC\mE\left[\left(\int_{k\Lambda}^{(k+1)\Lambda}
\left|\sigma(X^{t,x}_r)\right|^2dr\right)^q\right]\\
\leq&C\Lambda^{q-1}.
\end{align*}
{ This establishes the first inequality in  \eqref{e:4.21}. The second inequality in \eqref{e:4.21}  can be proved in a similar way.}
\end{proof}

\begin{proposition}\label{MMU1} 
{ Eq.\eqref{bsde2} has a unique progressively measurable solution $\left(U^{\varepsilon, t, x, y},  Z^{\varepsilon, t, x, y}\right)$.}
  In addition, for all $\lambda \geqslant 0$,  $(t, x) \in[0, T]\times\bar{D}_1$ and $q\geq1$, 
\be\label{MUZ0} 
\mathbb{E}  \Big[ \sup_{s \in[t, T]} e^{q\lambda \phi_s^{\varepsilon, t, x, y}}\left|U_s^{\varepsilon, t, x, y}\right|^q \Big] 
+E\left(\int_t^T e^{2\lambda \phi_s^{\varepsilon, t, x, y}}\left|Z_s^{\varepsilon, t, x, y}\right|^2 d s\right)^{q/2} <+\infty .
\ee
\end{proposition}

\begin{proof} The existence and uniqueness of Eq.\eqref{bsde2} can be proved by following arguments similar to \cite[Theorem 1.7] {par-zh} and  \cite[Proposition 2.5]{par-ras}. { By adapting an argument from \cite[Theorem 5.69]{P-R14},  we get} \eqref{MUZ0}  from Proposition \ref{MME1}.
\end{proof}

\begin{proposition}\label{MMU2}
For any $h>0$ and $s \in[t, T-h]$, and { any } $\lambda \geq 1$, 
\begin{eqnarray*}
\mathbb{E} \Big[ e^{\lambda\left(s+\phi_s^{\varepsilon,t,x,y}\right)}\left|U_s^{\varepsilon, t, x, y}-U_{s+h}^{\varepsilon, t, x, y}\right|^2 \Big]
+\mathbb{E} \int_s^{s+h} e^{\lambda\left(r+\phi_r^{\varepsilon,t,x,y}\right)}\left|Z_r^{\varepsilon, t, x, y}\right|^2 d r 
\leqslant Ch^{\frac12}.
\end{eqnarray*}
\end{proposition}

\begin{proof}
By It\^o's formula, 
\begin{align*}
&e^{\lambda\left(\phi_s^{\varepsilon, t, x, y}+s\right)}\left|U_s^{\varepsilon, t, x, y}-U_{s+h}^{\varepsilon, t, x, y}\right|^2 +\lambda \int_s^{s+h}e^{\lambda\left(\phi_r^{\varepsilon, t, x, y}+r\right)}\left|U_r^{\varepsilon, t, x, y}-U_{s}^{\varepsilon, t, x, y}\right|^2 \left(d r+d \phi_r^{\varepsilon, t, x, y}\right) \\
= &2 \int_s^{s+h} e^{\lambda\left(\phi_r^{\varepsilon, t, x, y}+r\right)}\langle U_r^{\varepsilon, t, x, y}-U_{s}^{\varepsilon, t, x, y}, f\left(X_r^{\varepsilon, t, x, y}, Y_r^{\varepsilon, t, x, y}, U_r^{\varepsilon, t, x, y}\right)\rangle d r 
\\
+& \int_s^{s+h} e^{\lambda\left(\phi_{r}^{\varepsilon, t, x, y}+r\right)}\langle U_r^{\varepsilon, t, x, y}-U_{s}^{\varepsilon, t, x, y},g\left(X_r^{\varepsilon, t, x, y}, U_r^{\varepsilon, t, x, y}\right)\rangle d \phi_r^{\varepsilon, t, x, y} \\
-&2 \int_s^{s+h} e^{\lambda\left(\phi_{r}^{\varepsilon, t, x, y}+r\right)}\langle U_r^{\varepsilon, t, x, y}-U_{s}^{\varepsilon, t, x, y}, Z_r^{\varepsilon, t, x, y} dW_r\rangle - \int_s^{s+h} e^{\lambda\left(\phi_{r}^{\varepsilon, t, x, y}+r\right)}\left|Z_r^{\varepsilon, t, x, y}\right|^2 d r\\
 \leq& 2\int_s^{s+h}e^{\lambda\left(\phi_{r}^{\varepsilon, t, x, y}+r\right)}\left|U_r^{\varepsilon, t, x, y}-U_{s}^{\varepsilon, t, x, y}\right|C_f\left(1+\left|X_r^{\varepsilon, t, x, y}\right|+\left|Y_r^{\varepsilon, t, x, y}\right|+\left|U_r^{\varepsilon, t, x, y}\right|\right)dr \\
+& 2 \int_s^{s+h} e^{\lambda\left(\phi_{r}^{\varepsilon, t, x, y}+r\right)}\left|U_r^{\varepsilon, t, x, y}-U_{s}^{\varepsilon, t, x, y}\right|\left|g\left(X_r^{\varepsilon, t, x, y}, U_r^{\varepsilon, t, x, y}\right)\right| d \phi_r^{\varepsilon, t, x, y}\\
-& 2 \int_s^{s+h} e^{\lambda\left(\phi_{r}^{\varepsilon, t, x, y}+r\right)}\langle U_r^{\varepsilon, t, x, y}-U_{s}^{\varepsilon, t, x, y}, Z_r^{\varepsilon, t, x, y}dW_r\rangle -\int_{s}^{s+h} e^{\lambda\left(\phi_{r}^{\varepsilon, t, x, y}+r\right)}\left|Z_r^{\varepsilon, t, x, y}\right|^2 d r \\
 \leq&\int_s^{s+h}e^{\lambda\left(\phi_{r}^{\varepsilon, t, x, y}+r\right)}\left|U_r^{\varepsilon, t, x, y}-U_{s}^{\varepsilon, t, x, y}\right|^2\left(d r+d \phi_r^{\varepsilon}\right) +C\int_s^{s+h}e^{\lambda\left(\phi_{r}^{\varepsilon, t, x, y}+r\right)}\left(1+\left|X_r^{\varepsilon, t, x, y}\right|^2+\left|Y_r^{\varepsilon, t, x, y}\right|^2+\left|U_r^{\varepsilon, t, x, y}\right|^2\right)dr\\
-&2 \int_s^{s+h} e^{\lambda\left(\phi_{r}^{\varepsilon, t, x, y}+r\right)}\langle U_r^{\varepsilon, t, x, y}-U_{s}^{\varepsilon, t, x, y}, Z_r^{\varepsilon, t, x, y} dW_r \rangle-\int_{s}^{s+h} e^{\lambda\left(\phi_{r}^{\varepsilon, t, x, y}+r\right)}\left|Z_r^{\varepsilon, t, x, y}\right|^2 d r\\
&+\int_s^{s+h} e^{\lambda\left(\phi_{r}^{\varepsilon, t, x, y}+r\right)}\left|g\left(X_r^{\varepsilon, t, x, y}, U_r^{\varepsilon, t, x, y}\right)\right|^2 d \phi_r^{\varepsilon, t, x, y}.
\end{align*}
Note that $\lambda\geq1$, we have 
 from the above that 
\begin{eqnarray*}
&&e^{\lambda\left(\phi_s^{\varepsilon, t, x, y}+s\right)}\left|U_s^{\varepsilon, t, x, y}-U_{s+h}^{\varepsilon, t, x, y}\right|^2 
  + \int_{s}^{s+h} e^{\lambda\left(\phi_{r}^{\varepsilon, t, x, y}+r\right)}\left|Z_r^{\varepsilon, t, x, y}\right|^2 d r  \\
& \leq & C\int_s^{s+h}e^{\lambda\left(\phi_{r}^{\varepsilon, t, x, y}+r\right)}\left(1+\left|X_r^{\varepsilon, t, x, y}\right|^2+\left|Y_r^{\varepsilon, t, x, y}\right|^2+\left|U_r^{\varepsilon, t, x, y}\right|^2\right)dr \\
&& -2 \int_s^{s+h} e^{\lambda\left(\phi_{r}^{\varepsilon, t, x, y}+r\right)}\langle U_r^{\varepsilon, t, x, y}-U_{s}^{\varepsilon, t, x, y}, Z_r^{\varepsilon, t, x, y} 
dW_r \rangle 
+\int_s^{s+h} e^{\lambda\left(\phi_{r}^{\varepsilon, t, x, y}+r\right)}\left|g\left(X_r^{\varepsilon, t, x, y}, U_r^{\varepsilon, t, x, y}\right)\right|^2 d \phi_r^{\varepsilon, t, x, y}.
\end{eqnarray*}
Then by Proposition \ref{MME1}-\ref{MMU1}, we have
\begin{align*}
&\mathbb{E} \left[ e^{\lambda\left(\phi_s^{\varepsilon, t, x, y}+s\right)}\left|U_s^{\varepsilon, t, x, y}-U_{s+h}^{\varepsilon, t, x, y}\right|^2 \right] 
+\mathbb{E} \int_{s}^{s+h} e^{\lambda\left(\phi_{r}^{\varepsilon, t, x, y}+r\right)}\left|Z_r^{\varepsilon, t, x, y}\right|^2 d r\\
\leq& C { \mathbb{E} }
\int_s^{s+h}e^{\lambda\left(\phi_{r}^{\varepsilon, t, x, y}+r\right)}\left(1+\left|X_r^{\varepsilon, t, x, y}\right|^2+\left|Y_r^{\varepsilon, t, x, y}\right|^2+\left|U_r^{\varepsilon, t, x, y}\right|^2\right)dr\\
&+C  { \mathbb{E} } \int_s^{s+h}e^{\lambda\left(\phi_{r}^{\varepsilon, t, x, y}+r\right)}\left(1+|X_r^{\varepsilon, t, x, y}|^2+|U_r^{\varepsilon, t, x, y}|^2\right)d\phi_r^{\varepsilon, t, x, y} \\
\leq& Ch^{\frac12}.
\end{align*}
This establishes the proposition. 
\end{proof}

Now we are in the position to prove Theorem \ref{HMbsde21}.

\medskip

\noindent{\it Proof of Theorem \ref{HMbsde21}.} 
 Similar to Proposition \ref{MMU1}-\ref{MMU2},  BSDE
 \eqref{bsdeav2} admits a unique strong solution {$(U^{t, x}, Z^{t, x})$ }, and   that for any $\lambda>0$ and $q\geq1$,
\begin{equation}\label{UZ10}
\mE \left[ \sup_{s \in[t, T]} e^{q\lambda \phi_s^{\varepsilon, t, x, y}}\left|U_s^{t, x}\right|^q \right]
+\mE 
{ \left[  \left(\int_t^T e^{2\lambda \phi_s^{\varepsilon, t, x, y}}\left|Z_s^{t, x}\right|^2 d s\right)^{q/2 } \right] }
<+\infty.
\end{equation}
Moreover, for any $\lambda>0$, 
\begin{equation}\label{UZ2}
\mathbb{E} \left[ e^{\lambda\left(s+\phi_s^{\varepsilon,t,x,y}\right)}\left|U_s^{t, x}-U_{s+h}^{t, x}\right|^2 \right]
+\mathbb{E} \int_s^{s+h} e^{\lambda\left(r+\phi_r^{\varepsilon,t,x,y}\right)}\left|Z_r^{t, x}\right|^2 dr 
\leq Ch^{\frac12}.
\end{equation}

     Divide $[t,T]$ into subintervals with the same length $\Lambda$. Denote $[s]:=k\Lambda$ for $s\in[k\Lambda,(k+1)\Lambda)$ for $k=0, 1, 2, \dots, [T/\Lambda]$. 
     { Let $\hat{Y}^{\e, y}$ be the solution determined by:
 \ce
\begin{cases}
  d\hat{Y}_s^{\e, y}=\e^{-1} b^{(2)}(X_{k\Lambda}^{\varepsilon, t, x, y},\hat{Y}_s^{\e, y})ds+\e^{-1/2}\sigma^{(2)}(X_{k\Lambda}^{\varepsilon, t, x, y},\hat{Y}_s^{\e, y})dB_s +\mathbf{n}^{(2)}(\hat{Y}_s^{\e, y})d\hat{\phi}^{\e,y}_s
  \quad \hbox{for } s\in [k\Lambda, (k+1)\Lambda),\medskip \\
 \hat{Y}_{t}^{\e, y}=y, \quad \hat{Y}_{(k+1)\Lambda}^{\e, y}=\lim_{s\to (k+1)\Lambda-}\hat{Y}_s^{\e, y}.
\end{cases}
\de
Then by \cite[Proposition 4.1]{C-W22}, 
there exists a constant $C$ independent of $\e, ~\Lambda$ such that for all $s\in[t,T]$, $x\in \bar D_1$ and $y\in \bar D_2$, 
\be\label{YhatY}
 \mE \left[ | \hat{Y}_s^{\e, y}-{Y}_s^{\varepsilon, t, x, y}|^2 \right] \leq C\Lambda.
 \ee 
It then follows from   \cite [Proposition  3.4] {C-W22} that 
\be\label{YhatY2}
\sup_{\e}\sup_{s\in[t,T]}\mE \left[ |\hat{Y}_s^{\e, y}|^2 \right]\leq C(1+|y|^2).
\ee
Following the arguments similar to \cite[(4.8)]{C-W22}, we get for any $r, s\in[t,T]$, $x\in \bar D_1$ and $y\in \bar D_2$, for any $q\geq1$,
\be\label{XbarXc}
\sup_{\varepsilon}\mE \left[  | X_s^{\varepsilon, t, x, y}-{X}_r^{\varepsilon, t, x, y}|^{2q} \right] \leq C|r-s|^{q}.
\ee

 For  $\eta\geq\max\{-2\alpha, 2C_0+1\}$,
applying It\^o's formula, for $s\in[t,T]$, we have:
\begin{align*}
&e^{\eta(\phi_s^{\varepsilon, t, x, y}+s)} |U^{\varepsilon,t, x, y}_{s}-U_s^{t,x}|^{2}+\eta\int_{s}^{T}e^{\eta(\phi_r^{\varepsilon, t, x, y}+r)} |{U}^{\varepsilon, t, x, y}_{r}
- U^{t,x}_{r}|^{2}(d\phi_r^{\varepsilon, t, x, y}+dr)\\
=&e^{\eta(\phi_T^{\varepsilon, t, x, y}+T)}|h(X^{\varepsilon, t, x, y}_{T})-h(X^{t,x}_{T})|^{2}\\
&+2\int_{s}^{T}e^{\eta(\phi_r^{\varepsilon, t, x, y}+r)}\langle  U^{\varepsilon, t, x, y}_{r}-U^{t,x}_{r},f(X^{\varepsilon, t, x, y}_{r},Y^{\varepsilon, t, x, y}_{r},{U}^{\varepsilon, t, x, y}_{r})-\bar{f}(X^{t,x}_{r},U^{t,x}_{r})\rangle dr\\
&+2\int_{s}^{T}e^{\eta(\phi_r^{\varepsilon, t, x, y}+r)}\langle  U^{\varepsilon, t, x, y}_{r}-U^{t,x}_{r},g(X^{\varepsilon, t, x, y}_{r},{U}^{\varepsilon, t, x, y}_{r})-g(X^{t,x}_{r},U^{t,x}_{r})\rangle d\phi^{\varepsilon, t, x, y}_{r}\\
&+2\int_{s}^{T}e^{\eta(\phi_r^{\varepsilon, t, x, y}+r)}\langle  U^{\varepsilon, t, x, y}_{r}-U^{t,x}_{r},g(X^{t,x}_{r},U^{t,x}_{r})\rangle (d\phi^{\varepsilon, t, x, y}_{r}-d\phi^{t, x}_{r})\\
&+2\int_{s}^{T}e^{\eta(\phi_r^{\varepsilon, t, x, y}+r)}\langle U^{\varepsilon,t, x, y}_{r}-U^{t,x}_{r},({Z}^{\varepsilon, t, x, y}_{r}-Z^{t,x}_{r})dW_{r}\rangle-\int_{s}^{T}e^{\eta(\phi_r^{\varepsilon, t, x, y}+r)} |{Z}^{\varepsilon, t, x, y}_{r}
- Z^{t,x}_{r}|^{2}dr\\
=:& \sum_{i=1}^6J_i.
\end{align*}
By Theorem \ref{HM4} and Proposition \ref{MME1}, 
$$
\begin{aligned}
\mathbb{E}J_1\leq& C_3\mE e^{\eta(\phi_T^{\varepsilon, t, x, y}+T)}|X^{\varepsilon, t, x, y}_{T}-X^{t,x}_{T}|^{2}\\
\leq &C_3\left\{\mE e^{2\eta(\phi_T^{\varepsilon, t, x, y}+T)}|X^{\varepsilon, t, x, y}_{T}-X^{t,x}_{T}|^{2}\right\}^{1/2}\left\{\mE |X^{\varepsilon, t, x, y}_{T}-X^{t,x}_{T}|^{2}\right\}^{1/2}
\leq C_T\varepsilon^{1/4}.
\end{aligned}
$$
By \eqref{MUZ0}-\eqref{UZ1},  $\mathbb{E}J_5=0$. By the condition on $g$ and Proposition \ref{MME1}-\ref{MMU1},
$$
\begin{aligned}
\mathbb{E}J_3\leq&-2\alpha\mE\int_{s}^{T}e^{\eta(\phi_r^{\varepsilon, t, x, y}+r)}\left|U^{\varepsilon,t, x, y}_{r}-U^{t,x}_{r}\right|^2d\phi_r^{\varepsilon, t, x, y}\\
&+2\mE\int_{s}^{T}e^{\eta(\phi_r^{\varepsilon, t, x, y}+r)}\left|U^{\varepsilon,t, x, y}_{r}-U^{t,x}_{r}\right|\left|X^{\varepsilon, t, x, y}_{r}-X^{t,x}_{r}\right|d\phi_r^{\varepsilon, t, x, y}\\
\leq&-2\alpha\mE\int_{s}^{T}e^{\eta(\phi_r^{\varepsilon, t, x, y}+r)}\left|U^{\varepsilon,t, x, y}_{r}-U^{t,x}_{r}\right|^2d\phi_r^{\varepsilon, t, x, y}\\
&+C\left\{\mE \Big[ \sup_{r\in[t,T]}\left|X^{\varepsilon, t, x, y}_{r}-X^{t,x}_{r}\right|^2 \Big] \right\}^{1/2}\left\{\mE\left(\int_{s}^{T}e^{\eta(\phi_r^{\varepsilon, t, x, y}+r)}\left|U^{\varepsilon,t, x, y}_{r}-U^{t,x}_{r}\right|d\phi_r^{\varepsilon, t, x, y}\right)^2\right\}^{1/2}\\
\leq&-2\alpha\mE\int_{s}^{T}e^{\eta(\phi_r^{\varepsilon, t, x, y}+r)}\left|U^{\varepsilon,t, x, y}_{r}-U^{t,x}_{r}\right|^2d\phi_r^{\varepsilon, t, x, y}+C\varepsilon^{1/4},
\end{aligned}$$
where Theorem \ref{HM4} is applied in last inequality. 
Note that for $J_2$,
\begin{equation*}
			\begin{aligned}
J_2=&2\int_{s}^{T}e^{\eta(\phi_r^{\varepsilon, t, x, y}+r)}\langle {U}^{\varepsilon,t, x, y}_{r}-U^{t,x}_r,f(X^{\varepsilon,t, x, y}_{r},Y^{\varepsilon,t, x, y}_{r},{U}^{\varepsilon,t, x, y}_{r})-\bar{f}(X^{t,x}_r,U^{t,x}_r)\rangle dr\\
\leq& 2\int_{s}^{T}e^{\eta(\phi_r^{\varepsilon, t, x, y}+r)}|U^{\varepsilon,t, x, y}_{r}-U^{t,x}_{r}|\left|f(X^{\varepsilon,t, x, y}_{r},Y^{\varepsilon,t, x, y}_{r},{U}^{\varepsilon,t, x, y}_{r})-f(X^{\varepsilon,t, x, y}_{[r]},\hat{Y}^{\varepsilon,y}_{r},{U}^{\varepsilon,t, x, y}_{[r]})\right|dr\\
&+2\int_{s}^{T}e^{\eta(\phi_r^{\varepsilon, t, x, y}+r)}\langle {U}^{\varepsilon,t, x, y}_{r}-U^{t,x}_r,\bar{f}(X^{\varepsilon,t, x, y}_{[r]},{U}^{\varepsilon,t, x, y}_{r})-\bar{f}(X^{\varepsilon,t, x, y}_{[r]},{U}^{t, x}_{r})\rangle dr\\
&+2\int_{s}^{T}e^{\eta(\phi_r^{\varepsilon, t, x, y}+r)}| {U}^{\varepsilon,t, x, y}_{r}-U^{t,x}_r|\left[\left|\bar{f}(X^{\varepsilon,t, x, y}_{[r]},{U}^{t, x}_{r})-\bar{f}(X^{t, x}_{r},{U}^{t, x}_{r})\right|+\left|\bar{f}(X^{\varepsilon,t, x, y}_{[r]},{U}^{\varepsilon,t, x, y}_{r})-\bar{f}(X^{\varepsilon,t, x, y}_{[r]},{U}^{\varepsilon,t, x, y}_{[r]})\right|\right]dr\\
&+2\int_{s}^{T}e^{\eta(\phi_r^{\varepsilon, t, x, y}+r)}\left|{U}^{\varepsilon,t, x, y}_{r}-U^{t,x}_r-U^{\varepsilon,t, x, y}_{[r]}+U^{t,x}_{[r]}\right|\left|{f}(X^{\varepsilon,t, x, y}_{[r]}, \hat{Y}_{r}^{\varepsilon, y}, {U}^{\varepsilon,t, x, y}_{[r]})-\bar{f}(X^{\varepsilon,t, x, y}_{[r]},U^{\varepsilon,t, x, y}_{[r]})\right|dr\\
&+2\int_{s}^{T}e^{\eta(\phi_r^{\varepsilon, t, x, y}+r)}\langle {U}^{\varepsilon,t, x, y}_{[r]}-U^{t,x}_{[r]},f(X^{\varepsilon,t, x, y}_{[r]},\hat{Y}_{r}^{\varepsilon,y}, U^{\varepsilon,t, x, y}_{[r]})-\bar{f}(X^{\varepsilon,t, x, y}_{[r]},{U}^{\varepsilon,t, x, y}_{[r]})\rangle dr\\
=:&\sum_{i=1}^5 J_{2,i}(s).
\end{aligned}\end{equation*}
By  Proposition \ref{MME1}-\ref{MMU2}, and \eqref{UZ10}-\eqref{YhatY2},
\ce
\mE J_{2,1}(s)&\leq&C\mE\left[\int_{s}^{T}e^{\eta(\phi_r^{\varepsilon, t, x, y}+r)} \big|U^{\varepsilon,t, x, y}_{r}-U^{t,x}_{r} \big| \,
\left(\left|X^{\varepsilon,t, x, y}_{r}-X^{\varepsilon,t, x, y}_{[r]}\right|+\left|Y^{\varepsilon,t, x, y}_{r}-\hat{Y}^{\varepsilon,y}_{r}\right|+\left|{U}^{\varepsilon,t, x, y}_{r}-{U}^{\varepsilon,t, x, y}_{[r]}\right|\right)dr\right]\\
&\leq& C\left\{\mE\int_{s}^{T}e^{2\eta(\phi_r^{\varepsilon, t, x, y}+r)} \left|U^{\varepsilon,t, x, y}_{r}-U^{t,x}_{r} \right|^2dr\right\}^{\frac12}
\left\{\mE\int_{s}^{T}\left(\left|X^{\varepsilon,t, x, y}_{r}-X^{\varepsilon,t, x, y}_{[r]}\right|^2+|Y^{\varepsilon,t, x, y}_{r}-\hat{Y}^{\varepsilon,y}_{r}|^2\right)dr\right\}^{\frac12}\\
&&+C\mE\int_{s}^{T}e^{\eta(\phi_r^{\varepsilon, t, x, y}+r)}\left|{U}^{\varepsilon,t, x, y}_{r}-{U}^{\varepsilon,t, x, y}_{[r]}\right|^2dr
+\frac12\mE\int_{s}^{T}e^{\eta(\phi_r^{\varepsilon, t, x, y}+r)} \left|U^{\varepsilon,t, x, y}_{r}-U^{t,x}_{r} \right|^2dr\\
&\leq& C\Lambda^{\frac12}+\frac12\mE\int_{s}^{T}e^{\eta(\phi_r^{\varepsilon, t, x, y}+r)} \left|U^{\varepsilon,t, x, y}_{r}-U^{t,x}_{r} \right|^2dr.\de
By the condition of $f$,
\ce
\mE J_{2,2}(s)\leq 2C_0\mE\int_{s}^{T}e^{\eta(\phi_r^{\varepsilon, t, x, y}+r)}\left|U^{\varepsilon,t, x, y}_{r}-U^{t, x}_{r}\right|^2dr.
\de
By  Theorem \ref{HM4}, Proposition \ref{MMU1}-\ref{MMU2}, and  \eqref{UZ10}-\eqref{YhatY2},
\ce
\mE J_{2,3}(s)&\leq&C\mE\left[\int_{s}^{T}e^{\eta(\phi_r^{\varepsilon, t, x, y}+r)} \big|U^{\varepsilon,t, x, y}_{r}-U^{t,x}_{r} \big| \,
\left(\left|X^{\varepsilon,t, x, y}_{[r]}-X^{t, x}_{r}\right|+\left|{U}^{\varepsilon,t, x, y}_{r}-{U}^{\varepsilon,t, x, y}_{[r]}\right|\right)dr\right]\\
&\leq&C\left\{\mE\int_{s}^{T}e^{2\eta(\phi_r^{\varepsilon, t, x, y}+r)} \left|U^{\varepsilon,t, x, y}_{r}-U^{t,x}_{r} \right|^2dr\right\}^{\frac12}
{\left\{\mE\int_{s}^{T}\left|X^{\varepsilon,t, x, y}_{[r]}-X^{t, x}_{r}\right|^2dr\right\}^{\frac12}}\\
&&{ +C\mE\int_{s}^{T}e^{\eta(\phi_r^{\varepsilon, t, x, y}+r)}\left|{U}^{\varepsilon,t, x, y}_{r}-{U}^{\varepsilon,t, x, y}_{[r]}\right|^2dr
+\frac12\mE\int_{s}^{T}e^{\eta(\phi_r^{\varepsilon, t, x, y}+r)} \left|U^{\varepsilon,t, x, y}_{r}-U^{t,x}_{r} \right|^2dr}\\
&\leq& { C\Lambda^{\frac12}+C\varepsilon^{\frac14}+\frac12\mE\int_{s}^{T}e^{\eta(\phi_r^{\varepsilon, t, x, y}+r)} \left|U^{\varepsilon,t, x, y}_{r}-U^{t,x}_{r} \right|^2dr}.\\
\mE J_{2,4}(s)&\leq&C\mE\int_{s}^{T}e^{\eta(\phi_r^{\varepsilon, t, x, y}+r)}\left(\left|{U}^{t,x}_{r}-U^{t,x}_{[r]}\right|+\left|{U}^{\e,t, x, y}_{[r]}-U^{\e,t, x, y}_r\right|\right)\left(1+\left|X^{\e,t, x, y}_{[r]}\right|+\left|\hat{Y}^{\e, y}_r\right|+\left|U^{\varepsilon, t, x, y}_{[r]}\right|\right)dr\\
&\leq&C\Lambda^{\frac14}.
\de
For the term $J_{2,5}$,
\begin{align*}
\mathbb{E}J_{2,5}(s)&=2\mE\int_{s}^{T}e^{\eta(\phi_r^{\varepsilon, t, x, y}+r)}\langle U^{\varepsilon,t, x, y}_{[r]}-U^{t,x}_{[r]},f(X^{\varepsilon,t, x, y}_{[r]},\hat{Y}^{\varepsilon,y}_{r},{U}^{\varepsilon,t, x, y}_{[r]})-\bar{f}(X^{\varepsilon,t, x, y}_{[r]},{U}^{\varepsilon,t, x, y}_{[r]})\rangle dr\\
\leq&2\mathbb{E}\sum_{k=0}^{[(T-s)/\Lambda]-1} \Big| \int_{k\Lambda }^{(k+1)\Lambda}
e^{\eta(\phi_r^{\varepsilon, t, x, y}+r)}\langle U^{\varepsilon,t, x, y}_{k\Lambda}-U^{t,x}_{k\Lambda}, f(X^{\varepsilon,t, x, y}_{k\Lambda},\hat{Y}^{\varepsilon,y}_{r}, {U}^{\varepsilon,t, x, y}_{k\Lambda})
-\bar{f}(X^{\varepsilon,t, x, y}_{k\Lambda}, {U}^{\varepsilon,t, x, y}_{k\Lambda})\rangle dr  \Big|
\nonumber\\
&+2\mathbb{E}\Big| \int_{[T]}^{T}
e^{\eta(\phi_r^{\varepsilon, t, x, y}+r)}\langle U^{\varepsilon,t, x, y}_{[r]}-U_{[r]}^{t,x},f(X^{\varepsilon,t, x, y}_{[r]},\hat{Y}^{\varepsilon,y}_{r},{U}^{\varepsilon,t, x, y}_{[r]})-\bar{f}(X^{\varepsilon,t, x, y}_{[r]},{U}^{\varepsilon,t, x, y}_{[r]})\rangle dr \Big|
\nonumber\\
=:&\mathbb{E}J_{2,5,1}(s)+ \mathbb{E}J_{2,5,2}(s).
\end{align*}
By  Proposition \ref{MME1}-\ref{MMU2}, and \eqref{UZ10}-\eqref{YhatY2},
\begin{align*}
\mathbb{E}J_{2,5,2}(s)&\leq C\mathbb{E}\int_{[T]}^{T}e^{\eta(\phi_r^{\varepsilon, t, x, y}+r)}
\left(\left|U^{\varepsilon,t, x, y}_{[r]}\right|+\left|U^{t,x}_{[r]}\right|\right) (1+\left|X^{\varepsilon,t, x, y}_{[r]}\right|+\left|\hat{Y}^{\varepsilon,y}_{r}\right|+\left|{U}^{\varepsilon,t, x, y}_{[r]}\right|)dr\leq C\Lambda.\\
 \mE J_{2,5,1}(s)&\leq C\Lambda^{ \frac14}+C\sum_{k=0}^{[\frac{T-s}{\Lambda}]-1}
\left(\mathbb{E}\left[e^{2\eta\phi_{k\Lambda}^{\varepsilon, t, x, y}}\left|U^{\varepsilon,t, x, y}_{k\Lambda}-U^{t,x}_{k\Lambda}\right|^2\right]\right)^{1/2}\\
&\cdot
\left(\mathbb{E}\Big|\int_{k\Lambda }^{(k+1)\Lambda }\left(f(X^{\varepsilon,t, x, y}_{k\Lambda},\hat{Y}^{\varepsilon,y}_{r}, {U}^{\varepsilon,t, x, y}_{k\Lambda})
-\bar{f}(X^{\varepsilon,t, x, y}_{k\Lambda}, {U}^{\varepsilon,t, x, y}_{k\Lambda})\right)dr\Big|^2\right)^{1/2}\\
&\leq C\Lambda^{\frac14}+\frac{C\varepsilon}{\Lambda} \max_{k}\left(\mathbb{E}\Big|\int_{0}^{\frac{\Lambda}{\varepsilon} }\left[f(X^{\varepsilon,t, x, y}_{k\Lambda},\hat{Y}^{\varepsilon,y}_{r\varepsilon+k\Lambda}, {U}^{\varepsilon,t, x, y}_{k\Lambda})
-\bar{f}(X^{\varepsilon,t, x, y}_{k\Lambda}, {U}^{\varepsilon,t, x, y}_{k\Lambda})\right]dr\Big|^2\right)^{1/2}\\
&\leq C\Lambda^{\frac14}+ \frac{C\varepsilon}{\Lambda} \max_{k}\Big(\mathbb{E}\int_{0}^{\frac{\Lambda}{\varepsilon}}\int_{r}^{\frac{\Lambda}{\varepsilon}}
\langle f(X^{\varepsilon,t, x, y}_{k\Lambda},\hat{Y}^{\varepsilon,y}_{s\varepsilon+k\Lambda}, {U}^{\varepsilon,t, x, y}_{k\Lambda})
-\bar{f}(X^{\varepsilon,t, x, y}_{k\Lambda}, {U}^{\varepsilon,t, x, y}_{k\Lambda}),  \\
& \hskip 1.6  truein   f(X^{\varepsilon,t, x, y}_{k\Lambda},\hat{Y}^{\varepsilon,y}_{r\varepsilon+k\Lambda}, {U}^{\varepsilon,t, x, y}_{k\Lambda})
-\bar{f}(X^{\varepsilon,t, x, y}_{k\Lambda}, {U}^{\varepsilon,t, x, y}_{k\Lambda})\rangle dsdr \Big)^{1/2}.
\end{align*}
Similar to that of \cite[Section 3]{C-W22}, it is easy to see that
\ce
\left\{\hat{Y}^{\e, y}_{k\Lambda+s};  \, s\in [0, \Lambda] \right\} \  \hbox{ has the same distribution as }  \
\Big\{ Y_{s/{\e}}^{X^{\e,t, x, y}_{k\Lambda},\hat{Y}^{\e,y}_{k\Lambda}}  ;  \, s\in [0, \Lambda] \Big\},
\de
 where $Y^{x,y}$ denotes the solution of the following frozen equation:
\ce
dY^{x,y}_r=\sigma^{(2)}(x, Y^{x,y}_r)dB_r+b^{(2)}(x,Y^{x,y}_r)dr- \mathbf{n}^{(2)}(Y^{x,y}_r) d\phi^{(2),x,y}_r\quad \hbox{with }
Y^{x,y}_0=y\in\bar{D}_2.
\de 
Applying the Markov property of $Y^{x,y}$ and using  \eqref{lefef2},   we have for $0<s<r<\frac{\Lambda}{\varepsilon}$,
\begin{align*}
&\mE\left[\langle f(X^{\varepsilon,t, x, y}_{k\Lambda},\hat{Y}^{\varepsilon,y}_{s\varepsilon+k\Lambda}, {U}^{\varepsilon,t, x, y}_{k\Lambda})
-\bar{f}(X^{\varepsilon,t, x, y}_{k\Lambda}, {U}^{\varepsilon,t, x, y}_{k\Lambda}), f(X^{\varepsilon,t, x, y}_{k\Lambda},\hat{Y}^{\varepsilon,y}_{r\varepsilon+k\Lambda}, {U}^{\varepsilon,t, x, y}_{k\Lambda})
-\bar{f}(X^{\varepsilon,t, x, y}_{k\Lambda}, {U}^{\varepsilon,t, x, y}_{k\Lambda})\rangle\right]\\
&=\mE\left[ \mE\left[\langle f(x,{Y}^{x,y}_{r}, u)
-\bar{f}(x, u), f(x,{Y}^{x,y}_{s},u)
-\bar{f}(x, u)\rangle\right]\Big|_{(x,y,u)=(X^{\varepsilon,t, x, y}_{k\Lambda},\hat{Y}^{\e}_{k\Lambda},U^{\varepsilon,t, x, y}_{k\Lambda})}\right] \\
&\leq C\mE\left(1+|X^{\varepsilon,t, x, y}_{k\Lambda}|+|\hat{Y}^{\e}_{k\Lambda}|+|U^{\varepsilon,t, x, y}_{k\Lambda}|\right)^2e^{-\gamma (s-r)/2}.
\end{align*}
Thus we have
\begin{align*}
\mathbb{E}J_{2,5,1}(s)\leq \frac{C\varepsilon}{\Lambda} \max_{k}\left(\mathbb{E}\left[\int_0^{\frac{\Lambda}{\varepsilon}}\int_r^{\frac{\Lambda}{\varepsilon}}
e^{-\gamma (s-r)/ 2} drds\right]\right)^{1/2}
\leq C_T (\Lambda^{1/4}+\sqrt{\varepsilon/ \Lambda}).
\end{align*}
For $J_4$, 
\begin{align*}
J_4(s) & =2 \int_s^T e^{\eta\left(\phi_r^{\varepsilon,t,x,y}+r\right)}\langle U_r^{\varepsilon,t, x, y}-U^{t,x}_r, g\left(X^{t,x}_r, U^{t,x}_r\right)\rangle\left(d \phi_r^{\varepsilon,t, x, y}-d \phi^{t,x}_r\right) \\
\leq & 2 \int_s^T e^{\eta\left(\phi_r^{\varepsilon,t,x,y}+r\right)}\left|U_r^{\varepsilon,t, x, y}-U_{[r]}^{\varepsilon,t, x, y}\right| \left|g\left(X^{t,x}_r,U^{t,x}_r\right)\right| d\left(\phi_r^{\varepsilon,t,x,y}+\phi_r^{t,x}\right) \\
& +2 \int_s^Te^{\eta\left(\phi_r^{\varepsilon,t,x,y}+r\right)}\left|U_{[r]}^{\varepsilon,t,x,y}-U^{t,x}_{[r]}\right|\left(\left|X^{t,x}_r-X^{t,x}_{[r]}\right|+\left|U^{t,x}_r-U^{t,x}_{[r]}\right|\right)\left(d \phi_r^{\varepsilon,t,x,y}+d \phi^{t,x}_r\right) \\
& +2 \int_s^T e^{\eta\left(\phi_r^{\varepsilon,t,x,y}+r\right)}\left\langle U^{\varepsilon,t, x, y}_{[r]}-U^{t,x}_{[r]}, g\left(X^{t,x}_{[r]}, U^{t,x}_{[r]}\right)\right\rangle d\left(\phi_r^{\varepsilon,t,x,y}-\phi_r^{t,x}\right)\\
&+2 \int_s^T e^{\eta\left(\phi_r^{\varepsilon,t,x,y}+r\right)}\left\langle U^{t,x}_{[r]}-U^{t,x}_r, g\left(X^{t,x}_{[r]}, U^{t,x}_{[r]}\right)\right\rangle d\left(\phi_r^{\varepsilon,t,x,y}-\phi_r^{t,x}\right)\\
=:&\sum_{i=1}^4J_{4,i}(s).
\end{align*}
By applying Proposition \ref{MME1}-\ref{MMU1}, for any $1<p<2$, 
\begin{align*}
\mE J_{4,1}(s)\leq&C\mE \int_s^T e^{\eta\left(\phi_r^{\varepsilon,t,x,y}+r\right)}\left|U_r^{\varepsilon,t, x, y}-U_{[r]}^{\varepsilon,t, x, y}\right| \left(1+|X^{t,x}_r|+|U^{t,x}_r|\right) d\left(\phi_r^{\varepsilon,t,x,y}+\phi_r^{t,x}\right)\\
\leq &C\left\{\mE \int_s^T \left|U_r^{\varepsilon,t, x, y}-U_{[r]}^{\varepsilon,t, x, y}\right|^p d\left(\phi_r^{\varepsilon,t,x,y}+\phi_r^{t,x}\right)\right\}^{\frac1p} \\ \cdot&\left\{\mE\int_s^T e^{\frac{p}{p-1}\eta\left(\phi_r^{\varepsilon,t,x,y}+r\right)}\left(1+|X^{t,x}_r|+|U^{t,x}_r|\right)^{\frac{p}{p-1}} d\left(\phi_r^{\varepsilon,t,x,y}+\phi_r^{t,x}\right)\right\}^{1-\frac1p}\\
\leq&C_{p,T}\left\{\mE \int_s^T \left|U_r^{\varepsilon,t, x, y}-U_{[r]}^{\varepsilon,t, x, y}\right|^p d\left(\phi_r^{\varepsilon,t,x,y}+\phi_r^{t,x}\right)\right\}^{\frac1p}.
\end{align*}
Similarly, 
\begin{align*}
\mE J_{4,4}(s)\leq&C_{p,T}\left\{\mE \int_s^T \left|U_r^{t, x}-U_{[r]}^{t, x}\right|^p d\left(\phi_r^{\varepsilon,t,x,y}+\phi_r^{t,x}\right)\right\}^{\frac1p}\\
&  \quad \times \left\{\mE \int_s^T e^{\frac{p}{p-1}\eta\left(\phi_r^{\varepsilon,t,x,y}+r\right)}\left|g\left(X^{t,x}_{[r]}, U^{t,x}_{[r]}\right)\right|^{\frac{p}{p-1}} d\left(\phi_r^{\varepsilon,t,x,y}-\phi_r^{t,x}\right)\right\}^{1-\frac1p}  \\
\leq&C_{p,T}\left\{\mE \int_s^T \left|U_r^{t, x}-U_{[r]}^{t, x}\right|^p d\left(\phi_r^{\varepsilon,t,x,y}+\phi_r^{t,x}\right)\right\}^{\frac1p};\\
\mE J_{4,2}(s)\leq&C\mE\int_s^Te^{\eta\left(\phi_r^{\varepsilon,t,x,y}+r\right)}\left|U_{[r]}^{\varepsilon,t,x,y}-U^{t,x}_{[r]}\right|\left(\left|X^{t,x}_r-X^{t,x}_{[r]}\right|+\left|U^{t,x}_r-U^{t,x}_{[r]}\right|\right)\left(d \phi_r^{\varepsilon,t,x,y}+d \phi^{t,x}_r\right) \\
\leq &C\left\{\mE\int_s^Te^{\frac{p}{p-1}\eta\left(\phi_r^{\varepsilon,t,x,y}+r\right)}\left|U_{[r]}^{\varepsilon,t,x,y}-U^{t,x}_{[r]}\right|^{\frac{p}{p-1}}\left(d \phi_r^{\varepsilon,t,x,y}+d \phi^{t,x}_r\right)\right\}^{\frac{p-1}{p}}\\
& \times \left( \left\{\mE\int_s^T\left|X^{t,x}_r-X^{t,x}_{[r]}\right|^p\left(d \phi_r^{\varepsilon,t,x,y}+d \phi^{t,x}_r\right)\right\}^{\frac1p}+\left\{\mE\int_s^T\left|U^{t,x}_r-U^{t,x}_{[r]}\right|^p\left(d \phi_r^{\varepsilon,t,x,y}+d \phi^{t,x}_r\right)\right\}^{\frac1p}\right) \\
\leq&C_{T,p}\left\{\mE\int_s^T\left|X^{t,x}_r-X^{t,x}_{[r]}\right|^p\left(d \phi_r^{\varepsilon,t,x,y}+d \phi^{t,x}_r\right)\right\}^{\frac1p}+C_{T,p}\left\{\mE\int_s^T\left|U^{t,x}_r-U^{t,x}_{[r]}\right|^p\left(d \phi_r^{\varepsilon,t,x,y}+d \phi^{t,x}_r\right)\right\}^{\frac1p}.
\end{align*}
According to Eq.\eqref{ave_rsde},
\ce
X^{t,x}_r-X^{t,x}_{[r]}=\int_{[r]}^r\bar{b}(X^{t,x}_u)du+\int_{[r]}^r\sigma(X^{t,x}_u)dW_u+\int_{[r]}^r\mathbf{n}(X^{t,x}_u)d\phi^{t,x}_u.
\de
We have 
\begin{align*}
&\mE\int_s^T\left|X^{t,x}_r-X^{t,x}_{[r]}\right|^p\left(d \phi_r^{\varepsilon,t,x,y}+d \phi^{t,x}_r\right)\\
\leq&C\mE\int_s^T\left|\int_{[r]}^r\bar{b}(X^{t,x}_u)du\right|^p\left(d \phi_r^{\varepsilon,t,x,y}+d \phi^{t,x}_r\right)+C\mE\int_s^T\left|\int_{[r]}^r\sigma(X^{t,x}_u)dW_u\right|^p\left(d \phi_r^{\varepsilon,t,x,y}+d \phi^{t,x}_r\right)\\
+&C\mE\int_s^T\left|\int_{[r]}^r\mathbf{n}(X^{t,x}_u)d\phi^{t,x}_u\right|^p\left(d \phi_r^{\varepsilon,t,x,y}+d \phi^{t,x}_r\right)\\
\leq &C\Lambda^{p/2}\mE\left[(\phi_T^{\varepsilon,t,x,y}+ \phi^{t,x}_T)\left(\int_0^T(1+|X^{t,x}_u|^2)du\right)^{\frac {p}{2}}\right]\\
+&C\sum_k\mE\left[\sup_{s\in[k\Lambda,(k+1)\Lambda]}\left|\int_{k\Lambda}^{s}\sigma(X^{t,x}_u)dW_u\right|^p\left(\phi^{\varepsilon,t,x,y}_{(k+1)\Lambda}
-\phi^{\varepsilon,t,x,y}_{k\Lambda}+\phi^{t,x}
_{(k+1)\Lambda}-\phi^{t,x}_{k\Lambda}\right)\right]\\
+&C\mE\int_s^T\left(\phi^{t,x}_r-\phi^{t,x}_{[r]}\right)^{p/2}
\left|\int_{[r]}^r\left|\mathbf{n}(X^{t,x}_u)\right|^2d\phi^{t,x}_u\right|^{p/2}\left(d \phi_r^{\varepsilon,t,x,y}+d \phi^{t,x}_r\right)\\
\leq&C\Lambda^{p/2}+C\sum_k\left\{\mE\left(\int_{k\Lambda}^{(k+1)\Lambda}\left|\sigma(X^{t,x}_u)\right|^2du\right)^{p}\right\}^{1/2}\left\{\mE\left(\phi^{\varepsilon,t,x,y}_{(k+1)\Lambda}
-\phi^{\varepsilon,t,x,y}_{k\Lambda}+\phi^{t,x}
_{(k+1)\Lambda}-\phi^{t,x}_{k\Lambda}\right)^{2}\right\}^{1/2}\\
+&C\mE\int_s^T(\phi^{t,x}_r-\phi^{t,x}_{[r]})^p\left(d \phi_r^{\varepsilon,t,x,y}+d \phi^{t,x}_r\right)\\
\leq& C\Lambda^{p/2}{+C\Lambda^{\frac{p-1}{2}}}+C\mE\left[\sum_k\left(\phi^{t,x}_{(k+1)\Lambda}
-\phi^{t,x}_{k\Lambda}\right)^p\left
(\phi^{\varepsilon,t,x,y}_{(k+1)\Lambda}
-\phi^{\varepsilon,t,x,y}_{k\Lambda}+\phi^{t,x}
_{(k+1)\Lambda}-\phi^{t,x}_{k\Lambda}\right)\right]\\
\leq&{C\Lambda^{\frac{p-1}{2}}}+C\left\{\mE\left[\max_k\left(\phi^{t,x}_{(k+1)\Lambda}
-\phi^{t,x}_{k\Lambda}\right)^{2p}\right]\mE\left[(\phi^{t,x}_T)^2  +(\phi^{\varepsilon,t,x,y}_T)^2 \right]\right\}^{1/2}\\
\leq &{C\Lambda^{\frac{p-1}{2}}},
\end{align*}
where we have used Proposition \ref{phiE1} in the last two inequalities.

Similarly, 
\begin{align*}&\mE\int_s^T\left|U^{t,x}_r-U^{t,x}_{[r]}\right|^p\left(d \phi_r^{\varepsilon,t,x,y}+d \phi^{t,x}_r\right)\\
\leq&C\mE\int_s^T\left[\left|\int_{[r]}^r\bar{f}(X^{t,x}_u,U^{t,x}_u)du\right|^p+\left|\int_{[r]}^rg(X^{t,x}_u,U^{t,x}_u)d\phi^{t,x}_u\right|^p+\left|\int_{[r]}^rZ^{t,x}_udW_u\right|^p\right]\left(d \phi_r^{\varepsilon,t,x,y}+d \phi^{t,x}_r\right)\\
\leq & C\Lambda^{p/2}\mE\int_s^T\left[\int_{[r]}^r\left|\bar{f}(X^{t,x}_u,U^{t,x}_u)\right|^2du\right]^{p/2}\left(d \phi_r^{\varepsilon,t,x,y}+d \phi^{t,x}_r\right)\\
+&C\mE\int_s^T\left(\phi^{t,x}_r-\phi^{t,x}_{[r]}\right)^{p/2}
\left(\int_{[r]}^r\left|g(X^{t,x}_u,U^{t,x}_u)\right|^2d\phi^{t,x}_u\right)^{p/2}\left(d \phi_r^{\varepsilon,t,x,y}+d \phi^{t,x}_r\right)\\
+&C\mE\int_s^T\left|\int_{[r]}^rZ^{t,x}_udW_u\right|^p\left(d \phi_r^{\varepsilon,t,x,y}+d \phi^{t,x}_r\right)\\
\leq&C\Lambda^{p/2}\left[\mE\left(\phi_T^{\varepsilon,t,x,y}+ \phi^{t,x}_T\right)^{\frac{2}{2-p}}\right]^{\frac{2-p}{2}}\left[\mE\int_s^T(1+|X^{t,x}_u|^2+|U^{t,x}_u|^2)du\right]^{p/2}\\
+&\mE\left[\left(\int_s^T|g(X^{t,x}_u,U^{t,x}_u)|^2d\phi^{t,x}_u\right)^{p/2}\sum_k\int_{k\Lambda}^{(k+1)\Lambda}(\phi^{t,x}_{r}-\phi^{t,x}_{k\Lambda})^{p/2}\left(d \phi_r^{\varepsilon,t,x,y}+d \phi^{t,x}_r\right)\right]\\
+&C\sum_k\left\{\mE\int_{k\Lambda}^{(k+1)\Lambda}|Z^{t,x}_u|^2du\right\}^{\frac{p}{2}}\left[\mE\left(\phi^{\varepsilon,t,x,y}_{(k+1)\Lambda}
-\phi^{\varepsilon,t,x,y}_{k\Lambda}+\phi^{t,x}_{(k+1)\Lambda}
-\phi^{t,x}_{k\Lambda}\right)^{\frac{2}{2-p}}\right]^{\frac{2-p}{2}}\\
\leq &C\Lambda^{p/2}+C\left[\mE\left(\sum_k(\phi^{t,x}_{(k+1)\Lambda}
-\phi^{t,x}_{k\Lambda})^{\frac{p}{2}}\left(\phi^{\varepsilon,t,x,y}_{(k+1)\Lambda}
-\phi^{\varepsilon,t,x,y}_{k\Lambda}+\phi^{t,x}
_{(k+1)\Lambda}-\phi^{t,x}_{k\Lambda}\right)\right)^{\frac{2}{2-p}}\right]^{\frac{2-p}{2}}\\
+&C\left\{\mE\int_{s}^{T}|Z^{t,x}_u|^2du\right\}^{\frac{p}{2}}\left[\sum_k\mE\left(\phi^{\varepsilon,t,x,y}_{(k+1)\Lambda}
-\phi^{\varepsilon,t,x,y}_{k\Lambda}+\phi^{t,x}_{(k+1)\Lambda}
-\phi^{t,x}_{k\Lambda}\right)^{\frac{2}{2-p}}\right]^{\frac{2-p}{2}}\\
\leq&C\Lambda^{p/2}+C\left\{\mE\left[\max_k\left(\phi^{t,x}_{(k+1)\Lambda}
-\phi^{t,x}_{k\Lambda}\right)^{\frac{p}{2}}(\phi^{\varepsilon,t,x,y}_T+\phi^{t,x}_T)\right]^{\frac{2}{2-p}}\right\}^{\frac{2-p}{2}}\\
+&C\mE\left[\max_k\left(\phi^{\varepsilon,t,x,y}_{(k+1)\Lambda}
-\phi^{\varepsilon,t,x,y}_{k\Lambda}+\phi^{t,x}_{(k+1)\Lambda}
-\phi^{t,x}_{k\Lambda}\right)^{\frac{p}{2-p}}(\phi^{\varepsilon,t,x,y}_T+\phi^{t,x}_T)\right]^{\frac{2-p}{2}}\\
\leq&C\Lambda^{p/2}+C\left\{\mE\left[\max_k\left(\phi^{t,x}_{(k+1)\Lambda}
-\phi^{t,x}_{k\Lambda}\right)^{\frac{2p}{2-p}}\right]\right\}^{\frac{2-p}{4}}\\
+&C\left\{\mE\max_k\left(\phi^{\varepsilon,t,x,y}_{(k+1)\Lambda}
-\phi^{\varepsilon,t,x,y}_{k\Lambda}+\phi^{t,x}_{(k+1)\Lambda}
-\phi^{t,x}_{k\Lambda}\right)^{\frac{2p}{2-p}}\right\}^{\frac{2-p}{4}}\\
\leq&C\Lambda^{\frac{p-1}{2}},
\end{align*}
where Proposition \ref{phiE1} is used in the last inequality. 

Similarly, by Proposition \ref{phiE1}-\ref{MMU1}, we have
\begin{align*}
\mE\int_s^T\left|U^{\varepsilon,t,x,y}_r-U^{\varepsilon,t,x,y}_{[r]}\right|^p\left(d \phi_r^{\varepsilon,t,x,y}+d \phi^{t,x}_r\right)\leq C\Lambda^{\frac{p-1}{2}},
\end{align*}
By Proposition \ref{MME1}-\ref{MMU1}, and \eqref{UZ10}, 
\begin{align*}
\mE J_{4,3} =
&2\mE \int_s^T e^{\eta\left(\phi_r^{\varepsilon,t,x,y}+r\right)}\left\langle U^{\varepsilon,t, x, y}_{[r]}-U^{t,x}_{[r]}, g\left(X^{t,x}_{[r]}, U^{t,x}_{[r]}\right)\right\rangle d\left(\phi_r^{\varepsilon,t,x,y}-\phi_r^{t,x}\right)\\
\leq&C \mathbb{E} \int_{[T]}^{T} e^{ \eta ( \phi_r^{\varepsilon,t,x,y} + r ) }{|U_{[r]}^{\varepsilon,t, x, y} - U_{[r]}^{t,x}| (1 + |X_{[r]}^{t,x}| + |U_{[r]}^{t,x}|) }
 d(\phi_r^{\varepsilon,t,x,y} + \phi_r^{t,x}) \\
& + \mathbb{E} \left[ e^{ \eta ( \phi_T^{\varepsilon,t,x,y} + T ) }\sum_k  |U_{k\Lambda}^{\varepsilon,t, x, y} - U_{k\Lambda}^{t,x}| (1 + |X_{k\Lambda}^{t,x}| + |U_{k\Lambda}^{t,x}|) \left(\phi_{(k+1)\Lambda}^{\varepsilon,t,x,y} - \phi_{k\Lambda}^{\varepsilon,t,x,y} - \phi_{(k+1)\Lambda}^{t,x} + \phi_{k\Lambda}^{t,x}\right) \right]\\
\leq &C_{T}\Lambda + C_T \left\{ \mathbb{E} \left[ \sup_{s\in[t,T]} \left(e^{ \eta\phi_s^{\varepsilon,t,x,y} }\left|U_{s}^{\varepsilon,t, x, y} - U_{s}^{t,x}\right|(1 +\left|X_{s}^{t,x}\right| + \left|U_{s}^{t,x}\right|)\right)^{\frac {2}{p}}\right] \right\}^{p/2}\\
& \times \left\{ \mathbb{E} \left[ \left| \sum_{k} \left( \phi_{(k+1)\Lambda}^{\varepsilon,t,x,y} - \phi_{k\Lambda}^{\varepsilon,t,x,y} - \phi_{(k+1)\Lambda}^{t,x} + \phi_{k\Lambda}^{t,x} \right) \right|^{\frac{2}{2-p}} \right] \right\}^{\frac{2-p}{2}}\\
\leq&C_{T}\Lambda+C_{T,p,\eta} \left\{ \mathbb{E} \left[ \max_{k} \left| \phi_{(k+1)\Lambda}^{\varepsilon,t,x,y} - \phi_{k\Lambda}^{\varepsilon,t,x,y} - \phi_{(k+1)\Lambda}^{t,x} + \phi_{k\Lambda}^{t,x} \right|^{\frac{p}{2-p}}\left(\phi_T^{\varepsilon,t,x,y}+\phi_T^{t,x}\right)\right] \right\}^{\frac{2-p}{2}}\\
\leq &C_{T}\Lambda+C_{T,p,\eta} \left\{ \mathbb{E} \left[ \max_{k} \left| \phi_{(k+1)\Lambda}^{\varepsilon,t,x,y} - \phi_{k\Lambda}^{\varepsilon,t,x,y} - \phi_{(k+1)\Lambda}^{t,x} + \phi_{k\Lambda}^{t,x} \right|^{\frac{2p}{2-p}}\right]\right\}^{\frac{2-p}{4}}\left\{ \mathbb{E}\left(\phi_T^{\varepsilon,t,x,y}+\phi_T^{t,x}\right)^2 \right\}^{\frac{2-p}{4}}\\
\leq &C_{T,\eta,p}\Lambda^{\frac{p-1}{2}}.
\end{align*}

Hence summing up the above estimates we get that for every $s\in [t, T]$, $x\in \bar D_1$ and $y\in \bar D_2$, 
\ce
\begin{aligned}
\mE \left[ e^{\eta\left(\phi_s^{\varepsilon,t,x,y}+s\right)}|U^{\varepsilon,t, x, y}_{s}-U^{t,x}_s \, |^{2} \right] + \mE \int_{s}^{T}e^{\eta\left(\phi_r^{\varepsilon,t,x,y}+r\right)}|{Z}^{\varepsilon,t, x, y}_{r}-Z^{t,x}_{r}|^{2}dr
\leq C_T \left(  \sqrt{\varepsilon/ \Lambda} +\varepsilon^{\frac14}+\Lambda^{\frac{p-1}{2p}} \right).
\end{aligned}\de
Recall that $p\in(1,2)$. We take $\Lambda=\eps^{\frac{p}{2p-1}}$ and thus we get 
\ce
\mE \left[e^{\eta\left(\phi_s^{\varepsilon,t,x,y}+s\right) }|U^{\varepsilon,t, x, y}_{s}-U^{t,x}_s|^{2} \right] + \mE \int_{s}^{T}e^{\eta\left(\phi_r^{\varepsilon,t,x,y}+r\right)}|{Z}^{\varepsilon,t, x, y}_{r}-Z^{t,x}_{r}|^{2}dr\leq C\e^{\frac{p-1}{2(2p-1)}}.
\de}
\eqref{hmbrate} now follows from the above estimate. \qed

\subsection{{ Quantitative} homogenization for PDEs with { zero} 
Neumann boundary conditions on nonsmooth domains}\label{S:4.3}

When $g\equiv0$, we can relax the conditions on the domains and assume $D_1$ and $D_2$ are two closed, convex, possibly unbounded domains with nonempty interiors  and consider the following PDE with zero Neumann boundary conditions:
\be\label{pdeN22}
\left\{
  \begin{array}{lllll}
    -\frac{\partial  }{\partial t}  u^{\e} =(\cL_1+ \e^{-1} \cL_2)u^{\e}+f(x,y,u^{\e})
     \quad \hbox{for } (t,x,y)\in [0,T]\times \bar D_1\times\bar{D}_2, \medskip\\
        \frac{\partial_x }{\partial \mathbf{n}}  u^{\e} =0     \quad  \hbox{for }   (t,x,y)\in [0,T]\times \partial D_1 \times\bar{D}_2,\medskip\\ 
       \frac{\partial_y }{\partial \mathbf{n}^{(2)}} u^{\e} =0 \quad   \hbox{for }  (t,x,y)\in [0,T]\times \bar{D}_1 \times\partial {D}_2,  \medskip\\
       u^{\e}(T,x,y)=h(x) \quad  \hbox{for }  ( x, y) \in \bar{D}_1 \times \bar D_2.
  \end{array}
\right.
\ee
To investigate the homogenization of this system, we consider the following backward SDE  decoupled with slow-fast reflected SDE system on $D_1$ and $D_2$:
\begin{equation}\label{bsde22}
U_s^{\varepsilon,t, x, y} =h(X_T^{\varepsilon,t, x, y}) + \int_s^T f(X_r^{\varepsilon,t, x, y}, Y_r^{\varepsilon,t, x, y}, U_r^{\varepsilon,t, x, y}) \, dr
 - \int_s^T Z_r^{\varepsilon,t, x, y} \, dW_r, 
\quad s\in [t, T], 
\end{equation}
and {$(X^{\varepsilon,t, x, y}, Y^{\varepsilon,t, x, y})$}
 are solutions of the following slow-fast reflected SDE system:
\be\label{rsde-bsde} 
\begin{cases}
dX_s^{\eps, t, x, y} = b(X_s^{\eps, t, x, y}, Y_s^{\eps, t, x, y}) \, ds + \sigma(X_s^{\eps, t, x, y}) \, dW_s 
+\mathbf{n}(X_s^{\eps, t, x, y} )d\phi^{\eps, t, x, y}_s,  \\
dY_s^{\eps, t, x, y} = \eps^{-1}b^{(2)}(X_s^{\eps, t, x, y}, Y_s^{\eps, t, x, y} ) \, ds 
+ \eps^{-1/2} \sigma^{(2)}(X_s^{\eps, t, x, y}, Y_s^{\eps, t, x, y}) \, dB_s  
+ \mathbf{n}^{(2)}(Y_s^{\eps, t, x, y}) d \phi_s^{(2), \eps, t, x, y},  \end{cases}
\ee
with $ X_t^{\eps, t, x, y}  = x\in \bar{D}_1$ and $Y_t^{\eps, t, x, y}  = y\in\bar{D}_2$, 
and as in Section 4.2, $B$ and $W$ are two independent Brownian motions.

According to \cite[Theorem 4.3.1]{ZJF}, \eqref{bsde2} has a unique solution  $(U^{\varepsilon,t, x, y}, Z^{\varepsilon,t, x, y})$ that are   progressively measurable with respect to the natural  
filtration generated by $\{\tilde{W}_r-\tilde{W}_t; r\in[t,s]\}_{t\leq s\leq T}$ and thus $U_t^{\varepsilon,t, x, y}$ is deterministic for every $t\in [0, T]$.
Moreover, for every $s \in [t,T]$, 
\begin{equation}\label{XY12}  
\mE \left[ |Y_s^{\varepsilon,t, x, y} |^2 \right] \leqslant C (1 + |x|^2 + |y|^2),   \quad
\mE  \left[ \sup_{t \leqslant s \leqslant T} |X_s^{\varepsilon,t, x, y}|^2 \right] \leqslant C(1 + |x|^2 + |y|^2).
\end{equation}
And 
\begin{equation}\label{UeZe12}  
\mE \left[  \sup_{s \in [t,T]} |U_s^{\varepsilon,t, x, y}|^2  \right] + \mE \int_t^T |Z_s^{\varepsilon,t, x, y}|^2 \, ds \leqslant C (1 + |x|^2 + |y|^2) < +\infty.
\end{equation}

We have the following result as the main result of this section, which is concentrated on the homogenization for the solution of Eq.\eqref{bsde22}.
\begin{theorem}\label{HMbsde222}
Suppose the conditions of Theorem \ref{HM4}  hold and that $f$ satisfies conditions in \textbf{(H.3)}. Then for every $s\in[t,T]$, 
{$x\in \bar D_1$ and $y\in \bar D_2$, }
$$
\mathbb{E} \left[ |U^{\varepsilon,t, x, y}_s-U^{t, x}_s|^2 \right] \leq C\varepsilon^{1/2},
$$
where 
 {$(U^{t, x}, Z^{t, x})$} is the unique solution to the backward SDE:
 \be\label{bsde2b2} 
U^{t, x}_s =h(X^{t, x}_T) + \int_s^T \bar{f}(X^{t, x}_r, U^{t, x}_r) \, dr - \int_s^T Z^{t, x}_r \, dW_r
\quad \hbox{for }  s\in [t,T].
\ee
\end{theorem}  

To prove Theorem \ref{HMbsde222}, we first present an estimate inspired 
from a stability result in \cite[Theorem 2.1]{HP97}
for the following backward SDE with random generator $F(s,v,z)$:
\begin{equation}\label{stabililyBSDE}
    V_s = \xi + \int_s^T F(r, V_r, Z_r) dr - \int_s^T Z_r dW_r, 
    \quad  s\in[0,T],
\end{equation}
where $W$ \textit{is a $d$-dimensional standard Brownian motion defined on} $(\Omega, \mathcal{F}, \mathbb{P})$ with filtration $\{\mathcal{F}_t\}_{t\in[0,T]}$, $\xi\in\mathcal{F}_T$ and $\mathbb{E}|\xi|^2<\infty$, $F :  \Omega \times [0, T] \times  \mathbb{R}^m \times \mathbb{R}^{m \times n} \to \mathbb{R}^m$. \textit{For any fixed $v \in \mathbb{R}^m, z \in \mathbb{R}^{m \times n}$, $F(\cdot, v, z)$ is $\{\mathcal{F}_t\}$-progressively measurable  satisfying}
$
\mathbb{E} \int_0^T |F(t, 0, 0)|^2 dt <\infty.
$

\begin{lemma}\label{stabilitylemma}
Suppose that  $F$ is uniformly Lipschitz continuous, that is, there exists a constant $L>0$ such that for any $s \in [0, T]$, 
$v_1, v_2\in\mathbb{R}^m, ~z_1, z_2\in\mathbb{R}^{m\times n}$,
$$
        |F(s, v_1, z_1) - F(s, v_2, z_2)| \leq L(|v_1 - v_2| + |z_1 - z_2|) 
        \quad \hbox{a.s.}.
 $$
Then the BSDE \eqref{stabililyBSDE} has a unique solution $(V, Z)$ and it satisfies that for every $s\in[0,T]$,
\begin{equation}\label{e:4.37}
   \mathbb{E} \left[|V_s|^2 + \int_s^T |Z_r|^2 dr \right]\leq C_T\left[\mathbb{E} [ |\xi|^2 ] +\sup_{t\in[s,T]}\mathbb{E}\left|\int_t^TF(r,0,0) dr\right|^2dt\right].
\end{equation}
\end{lemma}

\begin{proof}  By \cite[Theorems 4.3.1 and 4.4.4]{ZJF}, Eq. \eqref{stabililyBSDE} has a unique solution $(V,Z)$ satisfying 
$$
\mathbb{E} \Big[ \sup_{s\in[0,T]}|V_s|^2 \Big] +\mathbb{E}\int_0^T|Z_s|^2ds\leq C\mathbb{E}\left[|\xi|^2+\int_0^T|F(r,0,0)|^2dr\right].
$$
Note that for every $s\in[0,T]$, 
\begin{equation}\label{VZ1}
 V_s +\int_s^T Z_r dW_r= \xi +\int_s^TF(r,0,0) dr+ \int_s^T \left[F(r, V_r, Z_r)-F(r,0,0)\right] dr.
\end{equation}
By taking expectations of the square   on both sides,  we get 
$$
\begin{aligned}
&\mathbb{E}|V_s|^2+\mathbb{E}\int_s^T |Z_r|^2dr\\
&\leq  2\mathbb{E}\left[ \Big|\xi +\int_s^TF(r,0,0) dr\Big|^2 \right] +2\mathbb{E}\left[ \Big |\int_s^T \left[F(r, V_r, Z_r)-F(r,0,0)\right] dr\Big|^2 \right] \\
&\leq 2\mathbb{E} \left[ \Big| \xi +\int_s^TF(r,0,0) dr \Big|^2 \right]  +
2(T-s)L^2\mathbb{E}\int_s^T \left[|V_r|^2+|Z_r|^2\right] dr.
\end{aligned}
$$
Then we have for $s\in[T-\frac{1}{4L^2},T]$, 
\begin{equation}\label{e:4.38}
\mathbb{E}|V_s|^2+\frac12\mathbb{E}\int_s^T |Z_r|^2dr\leq 
4 \mathbb{E}\left[|\xi|^2 +\left|\int_s^TF(r,0,0) dr\right|^2\right]+ \frac12 \mathbb{E}\int_s^T |V_r|^2dr,
\end{equation}
For $t\in[0,T]$, define 
$$
g_t:=\mathbb{E}\left[|\xi|^2 +\left|\int_t^TF(r,0,0) dr\right|^2\right], \quad \tilde{g}_t:=g_{T-t},  
\quad \tilde{h}_t:=\mathbb{E} \left[ |V_{T-t}|^2 \right] .
$$ Then by \eqref{e:4.38},
$$
\tilde{h}_{s}\leq 4\tilde{g}_s+\frac12\int_0^s\tilde{h}_rdr.
$$
 Gronwall's lemma yields that 
 {
$$
\tilde{h}_{s}\leq 4 \tilde{g}_s+ 2 \int_0^se^{(s-r)/2 }\tilde{g}_rdr .
$$
Hence, 
$$
\mathbb{E} \left[ |V_s|^2 \right] \leq 4 g_s+ 2 \int_s^Tg_te^{(t-s)/2 } dt
\leq 4e^{T/2} \sup_{t\in[s,T]} g_t 
=4 e^{T/2}\sup_{t\in[s,T]}\mathbb{E}\left[|\xi|^2 +\left|\int_t^TF(r,0,0) dr\right|^2\right].
$$
}
 Putting this to the right hand side of \eqref{e:4.38} yields 
$$
\mathbb{E}\int_s^T |Z_r|^2dr\leq C_T\left[\mathbb{E}|\xi|^2 +\sup_{t\in[s,T]}\mathbb{E}\left|\int_t^TF(r,0,0) dr\right|^2\right].
$$

The estimate \eqref{e:4.37}
 then follows by induction on $[T-\frac{k+1}{4L^2},T-\frac{k}{4L^2}], ~k=0,1,\cdots$ as in the proof of \cite[Theorem 2.1]{HP97}.
\end{proof}

Now we are ready to prove   Theorem \ref{HMbsde222}. 
\begin{proof}
 Note that for the solutions $(U^{\varepsilon, t, x, y},Z^{\varepsilon, t, x, y})$, {$(U^{t, x}, Z^{t, x})$ } of \eqref{bsde22} and \eqref{bsde2b2}, analogous to Proposition \ref{MMU1} and \eqref{UZ10}, we have 
\begin{align}\label{UZ1}
&\sup_{\varepsilon}\mE \sup_{s \in [t,T]} |U^{\varepsilon, t, x, y}_s |^2 + \mE \int_t^T|Z^{\varepsilon, t, x, y}_s|^2 \, ds < +\infty.\nonumber\\
&\mE \sup_{s \in [t,T]} |U^{t,x}_s |^2 + \mE \int_t^T|Z^{t, x}_s|^2 \, ds < +\infty.
\end{align}
Applying It\^o's formula to $|U^{\varepsilon, t, x, y}_s-U^{\varepsilon, t, x, y}_{s-h}|^2$ for $t\leq s-h<s\leq T$,
$$
\begin{aligned}
|U_s^{\varepsilon, t, x, y} - U_{s-h}^{\varepsilon, t, x, y}|^2 
= &\,  2\int_{s-h}^s \langle  U_{r}^{\varepsilon, t, x, y} - U_s^{\varepsilon, t, x, y}, f(X_{r}^{\varepsilon, t, x, y}, Y_{r}^{\varepsilon, t, x, y}, U_{r}^{\varepsilon, t, x, y}) dr \rangle \\
& \quad + 2\int_{s-h}^s  \langle U_{r}^{\varepsilon, t, x, y} - U_s^{\varepsilon, t, x, y}, Z_{r}^{\varepsilon, t, x, y} dW_r \rangle - \int_{s-h}^s | Z_{r}^{\varepsilon, t, x, y} |^2 dr \\
\leq & \, \int_{s-h}^s |U_{r}^{\varepsilon, t, x, y} - U_s^{\varepsilon, t, x, y}|^2 dr 
+ \int_{s-h}^s |f(X_{r}^{\varepsilon, t, x, y}, Y_{r}^{\varepsilon, t, x, y}, U_{r}^{\varepsilon, t, x, y})|^2 dr \\
&\quad + 2\int_{s-h}^s  \langle U_{r}^{\varepsilon, t, x, y} - U_s^{\varepsilon, t, x, y}, Z_{r}^{\varepsilon, t, x, y} dW_r \rangle- \int_{s-h}^s | Z_{r}^{\varepsilon, t, x, y} |^2 dr.
\end{aligned}
$$ 
By taking expectations and by \eqref{UeZe12} and \eqref{UZ1}, we get for $t\leq s-h<s\leq T$, there exists a constant $C>0$ independent of $\varepsilon$ satisfying 
\be\label{shM}
\mathbb{E}\int_{s-h}^s| Z_{r}^{\varepsilon, t, x, y} |^2 dr + \mathbb{E} 
\left[ |U_s^{\varepsilon, t, x, y} - U_{s-h}^{\varepsilon, t, x, y}|^2 \right]  \leq Ch.
\ee
Similarly, for the solution  {$(U^{t, x}, Z^{t, x})$ } of \eqref{bsde2b}, 
 \be\label{shM2}
\mathbb{E}\int_{s-h}^s| Z_{r}^{t, x} |^2 dr + \mathbb{E} 
\left[ |U_s^{t, x} - U_{s-h}^{t, x}|^2 \right]  \leq Ch.
\ee
     Divide $[t,T]$ into subintervals with the same length $\Lambda$. Denote $[s]:=k\Lambda$ for $s\in[k\Lambda,(k+1)\Lambda)$ for $k=0, 1, 2, \dots, [T/\Lambda]$. 
      Let $\hat{Y}^{\e, y}$ be the solution determined by:
 \ce
\begin{cases}
  d\hat{Y}_s^{\e, y}=\e^{-1} b^{(2)}(X_{k\Lambda}^{\varepsilon, t, x, y},\hat{Y}_s^{\e, y})ds+\e^{-1/2}\sigma^{(2)}(X_{k\Lambda}^{\varepsilon, t, x, y},\hat{Y}_s^{\e, y})dB_s +\mathbf{n}^{(2)}(\hat{Y}_s^{\e, y})d\hat{\phi}^{\e,y}_s
  \quad \hbox{for } s\in [k\Lambda, (k+1)\Lambda),\medskip \\
 \hat{Y}_{t}^{\e, y}=y, \quad \hat{Y}_{(k+1)\Lambda}^{\e, y}=\lim_{s\to (k+1)\Lambda-}\hat{Y}_s^{\e, y}.
\end{cases}
\de
Then by \cite[Propositions { 3.4 and} 4.1]{C-W22},
there exists a constant $C$ independent of $\e, ~\Lambda$ such that for all $s\in[t,T]$, $x\in \bar D_1$ and $y\in \bar D_2$, 
\be\label{YhatYc}
\mE \left[ |\hat{Y}_s^{\e, y}|^2 \right]\leq C(1+|y|^2)
 \quad \hbox{and} \quad 
 \mE \left[ | \hat{Y}_s^{\e, y}-{Y}_s^{\varepsilon, t, x, y}|^2 \right] \leq C\Lambda.
 \ee 
By an argument similar to that for
\cite[(4.8)]{C-W22}, we get for any $r, s\in[t,T]$, $x\in \bar D_1$ and $y\in \bar D_2$, for any $q\geq1$,
\be\label{XbarX}  
\mE \left[  | X_s^{\varepsilon, t, x, y}-{X}_r^{\varepsilon, t, x, y}|^{2q} \right] \leq C|r-s|^{q}.
\ee
For $u\in\mathbb{R}$ and $s\in[t,T]$, set $$
\hat{F}(s,u):=f(X^{\varepsilon, t, x, y}_{s},Y^{\varepsilon, t, x, y}_{s},u+U^{t,x}_r)-\bar{f}(X^{t,x}_{r},U^{t,x}_{r}).
$$
Then $\hat{F}$ is a random function on $[t,T]\times\mathbb{R}$ so that for every $u\in\mathbb{R}$, $\hat{F}(\cdot, u)$ is progressively measurable, and 
that $(U^{\varepsilon,t, x, y}_s-U^{t,x}_s, Z^{\varepsilon,t, x, y}_s-Z^{t,x}_s; s\in[t,T])$ satisfies the following BSDE:
\be\label{hatUBSDE}
U^{\varepsilon,t, x, y}_s-U^{t,x}_s=h(X^{\varepsilon, t, x, y}_{T})-h(X^{t,x}_{T})
+\int_s^T\hat{F}(r,U^{\varepsilon,t, x, y}_{r}-U^{t,x}_r)dr-\int_s^T\left(Z^{\varepsilon,t, x, y}_r-Z^{t,x}_r\right)dW_{r}.
\ee
By the conditions on $f$ and $h$, and Proposition \ref{MME1},
\ce
\mathbb{E}\int_{t}^T|\hat{F}(s,0)|^2ds
&\leq& C\mathbb{E}\int_{t}^T[1+|X^{\varepsilon, t, x, y}_{s}|+|Y^{\varepsilon, t, x, y}_{s}|+|U^{t,x}_s|+|X^{t,x}_s|]^2ds\\
&\leq& C_T(1+|x|^2+|y|^2)<\infty.\\
\mathbb{E}\left|h(X^{\varepsilon, t, x, y}_{T})-h(X^{t,x}_{T})\right|^2
&\leq& C\mathbb{E}[1+|X^{\varepsilon, t, x, y}_{T}|+|X^{t,x}_T|^2]\leq C_T(1+|x|^2+|y|^2)<\infty.
\de
Then analogous to the proof of \cite[Theorem 3.9]{SWYZ}, we have by Lemma \ref{stabilitylemma}  that for every $s'\in[t,T]$,  
\begin{align*}
&\mathbb{E} \left[ |U^{\varepsilon,t, x, y}_{s'}-U_{s'}^{t,x}|^{2} \right] 
\leq C_T\mathbb{E} \left[ \left|h(X^{\varepsilon, t, x, y}_{T})-h(X^{t,x}_{T})\right|^2 \right]
+C_T\sup_{s\in[s',T]}\mathbb{E} \left[ \Big|\int_{s}^T \hat{F}(r,0)dr\Big|^2 \right] \\
\leq &C_T\mathbb{E} \left[ \left|h(X^{\varepsilon, t, x, y}_{T})-h(X^{t,x}_{T})\right|^2 \right]
+C_T\sup_{s\in[s',T]}\mathbb{E}\left[ \Big |\int_{s}^T \left( f(X^{\varepsilon, t, x, y}_r,Y^{\varepsilon, t, x, y}_r,U_r^{t,x})-\bar{f}(X^{t,x}_r,U^{t,x}_r)\right) dr\Big|^2 \right]\\
&\leq C\mathbb{E} \left[ |X^{\varepsilon, t, x, y}_{T}-X^{t, x}_{T}|^2 \right] 
+C_T\mathbb{E}\int_{s'}^T \left( f(X^{\varepsilon, t, x, y}_r,Y^{\varepsilon, t, x, y}_r,U_r^{t,x})-{f}(X^{\varepsilon, t, x, y}_{[r]},\hat{Y}^{\varepsilon,y}_r,U_{[r]}^{t,x})\right)^2dr\\
&+C\sup_{s\in[s',T]}\mathbb{E}\left|\int_{s}^T \left[{f}(X^{\varepsilon, t, x, y}_{[r]},\hat{Y}^{\varepsilon,y}_r,U_{[r]}^{t,x})-\bar{f}(X^{\varepsilon, t, x, y}_{[r]},U_{[r]}^{t,x})\right]dr\right|^2\\
&+C_T\mathbb{E}\int_{s'}^T \left[\bar{f}(X^{\varepsilon, t, x, y}_{[r]},U_{[r]}^{t,x})-\bar{f}(X^{t,x}_r,U^{t,x}_r)\right]^2dr\\
\leq & C\mathbb{E}|X^{\varepsilon, t, x, y}_{T}-X^{t, x}_{T}|^2+C\mathbb{E}\int_{s'}^T\left[|X^{\varepsilon, t, x, y}_{r}-X^{\varepsilon, t, x, y}_{[r]}|^2+|Y^{\varepsilon, t, x, y}_r-\hat{Y}^{\varepsilon,y}_r|^2+|U_{r}^{t,x}-U_{[r]}^{t,x}|^2\right]dr\\
&+C\sup_{s\in[s',T]}\mathbb{E}\left[ \Big|  \int_{s}^T \left( {f}(X^{\varepsilon, t, x, y}_{[r]},\hat{Y}^{\varepsilon,y}_r,U_{[r]}^{t,x})-\bar{f}(X^{\varepsilon, t, x, y}_{[r]},U_{[r]}^{t,x})\right) dr\Big|^2 \right] \\
&+C_T\mathbb{E}\int_{s'}^T\left( |X^{\varepsilon, t, x, y}_{[r]}-X^{t,x}_r|^2+|U^{t,x}_r-U^{t,x}_{[r]}|^2\right) dr\\
\leq &C\varepsilon^{1/2}+C\Lambda+C_T\sup_{s\in[s',T]}\mathbb{E}\left[ \Big| \int_{s}^T \left[{f}(X^{\varepsilon, t, x, y}_{[r]},\hat{Y}^{\varepsilon,y}_r,U_{[r]}^{t,x})-\bar{f}(X^{\varepsilon, t, x, y}_{[r]},U_{[r]}^{t,x})\right) dr\Big|^2 \right] ,
\end{align*}
where Theorem \ref{HM4} and \eqref{shM}-\eqref{XbarX} are applied in the last inequality.
Note that for every $s\in[t,T]$,
\begin{align*}
&\mathbb{E}\left[ \Big| \int_s^T \left( f(X^{\varepsilon, t, x, y}_{[r]},\hat{Y}^{\varepsilon,y}_r,U_{[r]}^{t,x})-\bar{f}(X^{\varepsilon, t, x, y}_{[r]},U_{[r]}^{t,x})\right)
dr\Big|^2\right]\\
\leq & 3\mathbb{E}\left[ \Big|  \int_{[T]}^T \left( {f}(X^{\varepsilon, t, x, y}_{[r]},\hat{Y}^{\varepsilon,y}_r,U_{[r]}^{t,x})-\bar{f}(X^{\varepsilon, t, x, y}_{[r]},U_{[r]}^{t,x})\right) dr\Big|^2\right] \\
&+3\mathbb{E}\left[ \Big| \int_{s}^{[s]+\Lambda} \left( {f}(X^{\varepsilon, t, x, y}_{[r]},\hat{Y}^{\varepsilon,y}_r,U_{[r]}^{t,x})-\bar{f}(X^{\varepsilon, t, x, y}_{[r]},U_{[r]}^{t,x})\right) dr\Big|^2\right] \\
&+3\mathbb{E} \left[ \Big|  \sum_{k=1+\frac{[s]}{\Lambda}}^{[T]/\Lambda}\int_{k\Lambda}^{(k+1)\Lambda} \left( {f}(X^{\varepsilon, t, x, y}_{k\Lambda},\hat{Y}^{\varepsilon,y}_r,U_{k\Lambda}^{t,x})-\bar{f}(X^{\varepsilon, t, x, y}_{k\Lambda},U_{k\Lambda}^{t,x})\right) dr\Big|^2\right] \\
=:& I_1+I_2+I_3.
\end{align*}
By Proposition \ref{MME1} and \eqref{YhatYc},
\be\label{I12}
I_i\leq C\Lambda(1+|x|^2+|y|^2),\quad i=1,2.
\ee
Note that
\begin{align*}
I_3&\leq \frac{C_T}{\Lambda}\sum_{k=1+\frac{[s]}{\Lambda}}^{[T]/\Lambda}
\mathbb{E}\left[ \Big| \int_{k\Lambda}^{(k+1)\Lambda} \left( f(X^{\varepsilon, t, x, y}_{k\Lambda},\hat{Y}^{\varepsilon,y}_r,U_{k\Lambda}^{t,x})-\bar{f}(X^{\varepsilon, t, x, y}_{k\Lambda},U_{k\Lambda}^{t,x})\right)dr\Big|^2 \right] \\
&\leq \frac{C_T\varepsilon^2}{\Lambda^2}\max_{k}\mathbb{E}\left[ \Big| \int_{0}^{\frac{\Lambda}{\varepsilon}}
 \left( {f}(X^{\varepsilon, t, x, y}_{k\Lambda},\hat{Y}^{\varepsilon,y}_{{ k\Lambda + \eps s}} ,U_{k\Lambda}^{t,x})
 -\bar{f}(X^{\varepsilon, t, x, y}_{k\Lambda},U_{k\Lambda}^{t,x})\right) {  ds} \Big|^2 \right]\\
& {    = \frac{2C_T\varepsilon^2}{\Lambda^2}  } 
\max_{k}\mathbb{E}  
\Big[  
\int_{0}^{\frac{\Lambda}{\varepsilon}}\int_{r}^{\frac{\Lambda}{\varepsilon}}  \Big\langle{f}(X^{\varepsilon, t, x, y}_{k\Lambda},\hat{Y}^{\varepsilon,y}_{s\varepsilon+k\Lambda},U_{k\Lambda}^{t,x})-\bar{f}(X^{\varepsilon, t, x, y}_{k\Lambda},U_{k\Lambda}^{t,x}),
    \\
  & \hskip 1.8truein  
  {f}(X^{\varepsilon, t, x, y}_{k\Lambda},\hat{Y}^{\varepsilon,y}_{r\varepsilon+k\Lambda}, 
   U_{k\Lambda}^{t,x})-\bar{f}(X^{\varepsilon, t, x, y}_{k\Lambda},U_{k\Lambda}^{t,x})\Big\rangle \, dsdr 
 \Big].
\end{align*}

Similar to that of \cite[Section 3]{C-W22}, it is easy to see that
\ce
\left\{\hat{Y}^{\e, y}_{k\Lambda+r};  \, r\in [0, \Lambda] \right\} \  \hbox{ has the same distribution as }  \
\Big\{ Y_{r/{\e}}^{X^{\e,t, x, y}_{k\Lambda},\hat{Y}^{\e,y}_{k\Lambda}}  ;  \, r\in [0, \Lambda] \Big\},
\de
 where $Y^{x,y}$ denotes the solution of the following frozen equation:
\ce
dY^{x,y}_r=\sigma^{(2)}(x, Y^{x,y}_r)dB_r+b^{(2)}(x,Y^{x,y}_r)dr
 +  \mathbf{n}^{(2)}(Y^{x,y}_r) d\phi^{(2),x,y}_r\quad \hbox{with }
Y^{x,y}_0=y\in\bar{D}_2.
\de 
Applying the Markov property of $Y^{x,y}$ and using  \eqref{lefef2},   we have for $0<s<r<\frac{\Lambda}{\varepsilon}$,
\begin{align*}
&\mE\left[\langle f(X^{\varepsilon,t, x, y}_{k\Lambda},\hat{Y}^{\varepsilon,y}_{s\varepsilon+k\Lambda}, {U}^{t, x}_{k\Lambda})
-\bar{f}(X^{\varepsilon,t, x, y}_{k\Lambda}, {U}^{t, x}_{k\Lambda}), f(X^{\varepsilon,t, x, y}_{k\Lambda},\hat{Y}^{\varepsilon,y}_{r\varepsilon+k\Lambda}, {U}^{t, x}_{k\Lambda})
-\bar{f}(X^{\varepsilon,t, x, y}_{k\Lambda}, {U}^{t, x}_{k\Lambda})\rangle\right]\\
&=\mE\left[ \mE\left[\langle f(x,{Y}^{x,y}_{r}, u)
-\bar{f}(x, u), f(x,{Y}^{x,y}_{s},u)
-\bar{f}(x, u)\rangle\right]\Big|_{(x,y,u)=(X^{\varepsilon,t, x, y}_{k\Lambda},\hat{Y}^{\e}_{k\Lambda},U^{t, x}_{k\Lambda})}\right] \\
&\leq C\mE(1+|X^{\varepsilon,t, x, y}_{k\Lambda}|+|\hat{Y}^{\e}_{k\Lambda}|+|U^{t, x}_{k\Lambda}|)^2e^{-\gamma (s-r)/2}.
\end{align*}
Thus we have by Proposition \ref{MME1} and \eqref{UZ1}, \eqref{YhatYc}, 
\begin{align*}
I_{3}\leq \frac{C\varepsilon^2}{\Lambda^2}(1+|x|+|y|)^2 \int_0^{\frac{\Lambda}{\varepsilon}}\int_r^{\frac{\Lambda}{\varepsilon}}
e^{-\gamma (s-r)/ 2} drds
\leq C_T(1+|x|^2+|y|^2)  \frac{\varepsilon}{\Lambda}.
\end{align*}
Hence we get that for every $s'\in [t, T]$, $x\in \bar D_1$ and $y\in \bar D_2$, 
\ce
\begin{aligned}
\mE \left[ |U^{\varepsilon,t, x, y}_{s'}-U^{t,x}_{s'} \, |^{2} \right] 
\leq C_T \left(\varepsilon^{1/2}+\Lambda+\frac{\varepsilon}{\Lambda} \right).
\end{aligned}\de
Taking $\Lambda=\eps^{1/2}$ leads to that for every $s'\in [t, T]$, $x\in \bar D_1$ and $y\in \bar D_2$,
\ce
\mE \left[ |U^{\varepsilon,t, x, y}_{s'}-U^{t,x}_{s'}|^{2} \right] + \mE \int_{s'}^{T}|{Z}^{\varepsilon,t, x, y}_{r}-Z^{t,x}_{r}|^{2}dr\leq C_{T,x,y}\e^{1/2}.
\de
This completes the proof of the theorem. 
\end{proof}

\br \label{r1}\rm
\begin{enumerate} [(i)]
\item  The convergence rate for $U^{\varepsilon, t,x,y}-U^{t,x,y}$   
in $\eps$ in Theorem \ref{HMbsde222} is the same as that
 for $X^{\varepsilon, x,y}-\bar{X}^x$ in Theorem \ref{HM4}.

\item The domains $D_1, ~D_2$ in Theorem \ref{HMbsde222} are only assumed to be convex.  They include the whole space and smooth convex open domains as special cases.  \qed
  \end{enumerate} 
\er

\br\label{r2}\rm
We would like to point out that \cite[Theorem 3.9]{SWYZ} considers the case when there is no reflection term involved in the forward SDE system  Eq.\eqref{rsde-bsde}, and \eqref{pdeN22} turns into a PDE without boundary conditions. In this case the authors establishes the $L^2$ homogenization rate $\epsilon^{1/2}$ for the homogenization, where the generator $f$ can depend on  $Z^{\varepsilon,t,x,y}$ as well.  When $f$ depends on $Z^{\varepsilon,t,x,y}$, estimates concerning the time difference of $Z^{\varepsilon,t,x,y}$ are needed and the situation is much more complicated when reflected terms and Neumann boundary conditions are involved as well. We leave its study to a future work. \qed
\er

Set $\bar{u}(t,x):=U^{t,x}_t$, where $U^{t,x}$ is the unique solution of \eqref{bsde2b}. Then $\bar{u} (t, x)$ is continuous 
 { on $[0, T]\times \bar D_1$} and, by \cite[Theorem 4.3]{par-zh},  is a viscosity solution to the following PDE:
\ce
\left\{
  \begin{array}{lllll}
    \frac{\partial }{\partial t}  \bar{u} + \bar{\cL}\bar{u}+\bar{f}(x,\bar{u})=0  \quad  \hbox{for }  (t,x)\in[0,T]\times D_1,  \medskip \\
       \frac{\partial  }{\partial \mathbf{n}} \bar{u} =0    \quad  \hbox{for }   (t,x)\in[0,T]\times\partial D_1,  \medskip
    \\
    \bar{u}(T,x)={h}(x)  \quad  \hbox{for }  x\in\bar{D}_1,
  \end{array}
\right.
\de
where $\bar \cL :=\sum_{i,j=1}^n{a}^{ij}(x)\frac{\partial ^2}{\partial x^i\partial x^j}
+ \sum_{i=1}^n  \bar{b}^i(x)\frac{\partial}{\partial x^i}$.
  Hence as a  { direct application} of Theorem \ref{HMbsde222}, a convergence rate can be obtained for the homogenization of the viscosity solution $u^{\eps}(t,x,y)$ of equation \eqref{pdeN22}.
  
\begin{corollary}\label{HM-pde-N3c}
Suppose the conditions of Theorem \ref{HMbsde222} hold. Then for every $T>0$,  there is a constant $C>0$ so that 
$$|u^{\e}(t,x,y)-\bar{u}(t,x)|\leq C\e^{1/4},
\quad   \hbox{for every } (t, x, y) \in[0,T]\times \bar{D}_1\times \bar{D}_2.
$$
\end{corollary}

\vskip 0.3truein

\nin \textbf{Acknowledgments:}
 The authors   thank the anonymous reviewer   for  helpful comments. 
  The first author is supported in part by a Simons Foundation grant. 
The second named author Jing Wu is supported in part by NSFC (No. 12471144).

\end{document}